\documentclass[]{interact}
\usepackage{indentfirst}
\usepackage{mathrsfs} 
\usepackage{amsmath} 
\usepackage{amsthm}	
\usepackage{multirow}
\usepackage{lscape}
\usepackage{algorithm}
\usepackage{algpseudocode}
\usepackage{graphicx}
\usepackage[justification=centering]{caption} 
\usepackage{url}
\usepackage{cite}
\usepackage{caption}
\usepackage{color}
\usepackage{epstopdf}
\usepackage[caption=false]{subfig}
\usepackage[doublespacing]{setspace}
\theoremstyle{plain}

\newtheorem{corollary}{Corollary}
\newtheorem{proposition}{Proposition}

\theoremstyle{definition}
\newtheorem{definition}{Definition}

\theoremstyle{remark}

\newtheorem*{notation}{Notation}

\begin{document}
\articletype{}
\title{Condition-based Maintenance for Multi-component Systems: Modeling, Structural Properties, and Algorithms}
\author{
	\name{Zhicheng Zhu and Yisha Xiang}
	\affil{Texas Tech University, Lubbock, TX, 79409, USA}
}
\maketitle

\begin{abstract}
 Condition-based maintenance (CBM) is an  effective maintenance strategy to improve system performance while lowering  operating and maintenance costs. Real-world systems typically consist of a large number of components with various interactions between components. However, existing studies on CBM focus on single-component systems.  Multi-component condition-based maintenance, which joins the components' stochastic degradation processes and the combinatorial maintenance grouping problem, remains an open issue in the literature. In this paper, we study the CBM optimization problem for multi-component systems. We first develop a multi-stage stochastic integer model with the objective of minimizing the total maintenance cost over a finite planning horizon. We then investigate the structural properties of a two-stage model. Based on the structural properties, two efficient algorithms are designed to solve the two-stage model. Algorithm 1 solves the problem to its optimality and Algorithm 2 heuristically searches for high-quality solutions based on Algorithm 1. Our computational studies show that Algorithm 1 obtains optimal solutions in a reasonable amount of time and Algorithm 2 can find high-quality solutions quickly. The multi-stage problem is solved using a rolling horizon approach based on the algorithms for the two-stage problem.
\end{abstract}


\section{Introduction} 
Reliability is the central concern of many mission-critical systems, such as aerospace systems, electric power systems, and nuclear systems. Investigations show that many accidents are caused by equipment failures, which were attributed to the lack of effective maintenance methods. For example, the space shuttle Challenger accident \cite{feynman1986report} and the Deepwater Horizon drilling rig explosion \cite{pate1993learning} occur in part because of inadequate maintenance. As the complexity of modern engineering systems increases, it is imperative to develop cost-effective maintenance plans for complex systems.

Maintenance strategies can be generally classified into two categories: time-based maintenance (TBM) and condition-based maintenance (CBM). The literature on TBM and CBM for single-component systems is abundant \cite{ding2015maintenance,alaswad2017review}. However, much less attention has been paid on multi-component systems. Existing studies on multi-component systems are mainly time-based \cite{dekker1997review, nicolai2008optimal, thomas1986survey}. Despite the fact that CBM can be more cost-effective compared to TBM \cite{ahmad2012overview,tian2011condition,alaswad2017review}, CBM for multi-component systems is underexplored.

A multi-component system is usually subject to various interactions among components, such as stochastic dependence, structural dependence, and economic dependence \cite{thomas1986survey,dekker1997review}. Stochastic dependence means the state of one component influences the lifetime distributions of other components. Structural dependence applies if components structurally form a part, so that maintenance of a failed component implies maintenance of other components as well. Economic dependence occurs if any maintenance action incurs a fixed system-dependent cost, often referred to as setup cost, due to mobilizing repair crew,  disassembling machines, and downtime loss \cite{laggoune2010impact}. This setup cost can be significant in many capital-intensive industries. For example,  the production losses during the shutdown ranges from \$500 to \$100,000 per hour in a chemical plant and millions of dollars per day in offshore drilling refineries \cite{amaran2016medium}. Therefore, significant cost savings can be achieved by  maintaining multiple components jointly instead of separately.

In this paper, we study the CBM optimization problem for multi-component systems with economic dependence over a finite planning horizon using a stochastic programming approach. The objective is to minimize total maintenance cost by selecting components for maintenance at each decision period. This problem is challenging because it joins the components' stochastic degradation processes and the combinatorial maintenance grouping problem \cite{dekker1997review,scarf1997application,dekker1998impact}. In addition, the component state transition probability depends on the maintenance decision, making the problem decision-dependent, which is different from the standard approach to formulating stochastic programs based on the assumption that the stochastic process is independent of the optimization decisions. This endogenous uncertainty can make the stochastic programs more computationally challenging. There is a lack of general methods to efficiently solve this type of problem. Existing studies on multi-component maintenance planning often use simplified assumptions \cite{tian2011condition,grall2002condition,castanier2005condition} or resort to simulation methods \cite{barata2002simulation,laggoune2009opportunistic} to reduce mathematical difficulties in modeling and solving this problem. 

We develop a general multi-stage stochastic maintenance model and do not restrict any grouping opportunities. Due to the complexity of the multi-stage stochastic maintenance model with integer decision variables, we first consider a two-stage model and investigate its structural properties. Based on the structural properties, we design two efficient algorithms to solve the two-stage model. The multi-stage model is solved using a rolling horizon approach based on the algorithms for the two-stage model. The main contribution of this paper is threefold.

(1)Develop an analytical CBM model for multi-component systems using a stochastic programming approach. This model is among the very first efforts that provide analytical expressions for the cost function and maintenance decisions of multi-component CBM. The proposed model is general with no restrictions for grouping as opposed to exiting one that only allow grouping at PM or CM.

(2)Establish structural properties for the two-stage model. These theoretical properties provide the conditions and search directions of improving any feasible solution, and lead to significant reduction of the search space of the problem.

(3) Design efficient algorithms to find high-quality solutions. We develop algorithms for the two-stage problem based on its structural properties, which are then implemented on a rolling-horizon to solve the multi-stage problem. Computational studies show that our algorithms can provide satisfactory solutions within a reasonable amount of time, particularly for large-scale problems.

The remainder of this paper is organized as follows. Section \ref{secLit} reviews the related studies on multi-component maintenance and stochastic programming methods. In Section \ref{secModel}, we develop a CBM model for multi-component systems over multiple decision periods. Section \ref{secProp} investigates structural properties for the two-stage model. In Section \ref{secAlg}, we design two algorithms for the two-stage problem and use rolling horizon technique  to approximate the multi-stage model.  Computational studies are presented in Section \ref{sectionCom}. Section \ref{sectionConclusion} concludes this research and discusses the future research directions.

\section{Literature review}
\label{secLit}
We model the CBM optimization problem for multi-component systems  using stochastic programming. We first examine the existing literature on multi-component maintenance, and then review the solution techniques for stochastic programming.

\subsection{Multi-component maintenance} 

Most studies on multi-component maintenance are focused on TBM, which can be further divided into direct-grouping \cite{dekker1996joint,wildeman1997dynamic}  and indirect-grouping approaches \cite{goyal1985determining,goyal1992determining,epstein1985opportunistic,hariga1994deterministic} . Direct-grouping approach partitions the components into several fixed groups and  always maintains the components in a group jointly. By using this approach, the problem becomes a set-partitioning problem, which is NP-complete. Indirect-grouping groups preventive maintenance (PM) activities by making the PM interval a multiple of a basis interval, so that the maintenance of different components can coincide \cite{goyal1985determining,goyal1992determining}, or performs major PM on all component jointly at the end of a common interval and allows minor or major PM within this interval \cite{epstein1985opportunistic,hariga1994deterministic}. Unlike the fixed structure under direct-grouping, there is no fixed group structure under indirect-grouping. Some researchers formulate an indirect-grouping model  as a mixed
integer programming (MIP) problem \cite{sule1979determination,epstein1985opportunistic,hariga1994deterministic}. Because of the simplified policy structure, the MIP model can be separated by
components, which greatly reduces the computational complexity.
However, both direct- and indirect-grouping approaches only group PM activities and ignore the grouping opportunities provided by CM. Patriksson et al. \cite{patriksson2015stochastic} use stochastic programming to model a time-based multi-component replacement problem. However, CM and PM have the same cost and are not distinguished in their paper. 

Much less attention has been paid to CBM for multi-component systems \cite{alaswad2017review}. Opportunistic maintenance (OM) has been considered for multi-component CBM \cite{grall2002condition,shafiee2015opportunistic,castanier2005condition}. OM takes advantage of CM by performing PM on functioning components when any failure happens. Castanier et al. \cite{castanier2005condition} consider both PM and CM as opportunities for maintaining other functioning components and formulate the problem as a semi-regenerative process. However, they only consider a two-component system because of the exponential growth of problem size. Some studies have used Markov decision process (MDP) to solve the multi-component maintenance optimization problem. However, due to the state space grows exponentially as the number of components and/or the number of states increase, this method is limited to small scale problems. For example, Jia \cite{jia2010structural}  models the OM problem  as an MDP and investigate the structural property of the optimal policy. A two-component system is studied in \cite{jia2010structural}. Several studies use simulation methods to find optimal opportunistic CBM policies \cite{barata2002simulation,laggoune2009opportunistic}, which also suffers from curse of dimensionality.

Proportional hazard model (PHM) incorporates both event data and CM data by modeling the lifetime of a component as a hazard rate process \cite{alaswad2017review,tsui2015prognostics}. Tian et al. extends the PHM from single-component CBM to multi-component CBM \cite{tian2011condition}. They study two practical cases with systems of two components and three components.

\subsection{Stochastic programming} 

Various methods and techniques have been developed to solve a stochastic programming problem. For a two-stage stochastic linear program, Benders decomposition \cite{birge2011introduction,bodur2016strengthened} and progressive hedging algorithm (PHA) \cite{rockafellar1991scenarios,watson2011progressive} are two major decomposition methods. Benders decomposition is a vertical decomposition approach that decomposes the problem into a master problem that consists of the first-stage decisions and the subproblems that consist of second-stage decisions of all scenarios. PHA is a horizontal decomposition approach that decomposes the problem by scenarios. It first independently solves all subproblems at each iteration and then forces the non-anticipatively constraints converge.

Multi-stage stochastic programming extends two-stage stochastic programming by allowing revised decisions at each stage based on uncertainty realizations observed so far \cite{ahmed2003multi}. For a multi-stage stochastic linear program, nested Benders decomposition \cite{naoum2010nested,parpas2007computational} that extended from Benders decomposition and PHA are also two common solution approaches. However, because the size of the scenario tree grows exponentially as the number of stages increases, both approaches are computationally intractable. Stochastic dual dynamic programming (SDDP) \cite{shapiro2011analysis,pereira1991multi} overcomes the exploding scenario tree size problem in nested Benders decomposition by combining scenario tree nodes. The drawback of the SDDP approach is that it relies on special problem structure such as stage-wise independence \cite{bodur2017two}. Rolling horizon provides a heuristic approach to approximating a  multi-stage stochastic program by solving the  two-stage problem on a rolling basis and utilizing the first-stage solution \cite{kouwenberg2001scenario,beraldi2011short,tolio2007rolling}. This approximation approach requires the two-stage problem to be computationally tractable. Recently, rule-based method has attracted some interests in addressing the intractability issue in multi-stage stochastic programming \cite{shapiro2005complexity,bampou2011scenario,chen2008linear,bodur2017two}. This method restricts the solution to have some specific function forms, such as linear \cite{bodur2017two}, piece-wise linear \cite{chen2008linear}, and polynomial \cite{bampou2011scenario}. Because the optimal decision rules of arbitrary multi-stage stochastic programs do not have general forms, rule-based methods cannot guarantee the solution quality in general \cite{shapiro2005complexity}.

A stochastic integer program further combines the difficulty of stochastic programming and integer programming and is challenging to solve. Nested Benders decomposition and SDDP that utilize Benders cuts become prohibited to this problem because strong duality does not hold due to integrality constraints. Moreover, PHA does not perform well for this problem in general because the non-anciticipativity constraints may converge slowly due to the integrality constraints and the intractability of solving each integer subproblem. Integer L-shaped method is another approach in solving stochastic integer program by using integer L-shaped cuts within the Benders decomposition framework \cite{birge2011introduction,laporte1993integer}. However, this method is typically inefficient because it needs to generate an integer L-shaped cut for every feasible solution in the worst case scenario.

Stochastic programming with endogenous uncertainty draws some attentions recently because this type of uncertainty presents in a large number of applications \cite{goel2006class,zhan2017generation,peeta2010pre}. Endogenous uncertainty implies that the underlying stochastic process is influenced by the decisions. Therefore, the probabilities of scenarios are decision-dependent and usually nonlinear \cite{zhan2017generation,peeta2010pre}. There is a lack of efficient method to solve this type of problem. 

Our review shows that there is no general method to solve the proposed multi-stage stochastic maintenance model with integer decision variables and endogenous uncertainty. Efficient algorithms are needed to find high-quality solutions.

\section{Model development}
\label{secModel}
\begin{notation} \quad
\begin{description}
	\setlength{\parskip}{0pt}
	\item [$n:$] number of components
	\item[$\mathcal{N}:$] component set, $\mathcal{N}=\{1, 2, ..., n\}$
	\item [$T:$] number of decision stages
	\item[$\mathcal{T}:$] decision-stage set, $\mathcal{T}=\{1, 2, ..., T\}$
	\item [$\Omega_t:$] node set at stage $t \in \mathcal{T}$
	\item [$\omega_t:$] index of node at stage $t \in \mathcal{T}$, i.e., $\omega_t \in \Omega_t$
	\item [$a(\omega_t):$] ancestor node of $\omega_t \in \Omega_t$, $t \in \mathcal{T}\backslash\{1\}$
	\item [$\Omega(\omega_t):$] child nodes of $\omega_t \in \Omega_t$, $t \in \mathcal{T}\backslash\{T\}$				
	\item[$g_{it}:$] state of component \textit{i} at stage \textit{t}
	\item[$g_{it}^{\omega_t}:$] state of component \textit{i} in scenario $\omega_t \in \Omega_t$  at stage $t \in \mathcal{T}$
	\item[$Q_i(g,g^\prime):$] state transition probability from state $g$ to $g^\prime$ for component $i \in \mathcal{N}$
	\item[$c_{i, {\rm pm}}:$] PM cost of component \textit{i}
	\item[$c_{i,{\rm cm}}:$] CM cost of component \textit{i}
	\item[$c_{\rm s}:$] setup cost
	\item[$\tilde{x}_{it}:$] maintenance decision of component \textit{i} at stage \textit{t} without considering economic dependence
	\item[$\tilde{x}_{it}^{\ast}:$] optimal maintenance decision of component $i \in \mathcal{N}$ at stage $t \in \mathcal{T}$ without considering economic dependence
	\item[$x_{it}:$] equals to 1 if any maintenance is performed on component  $i \in \mathcal{N}$ at stage $t \in \mathcal{T}$ and 0 otherwise
	\item[$x_{it}^{\omega_t}:$] $x_{it}$ in scenario $\omega_t \in \Omega_t$
	\item[$x_{t}:$] vector of $x_{it}$ for all $i \in \mathcal{N}$ at stage $t \in \mathcal{T}$: $(x_{1,t}, x_{2,t}, ..., x_{n,t})$
	\item[$x_{t}^{\omega_t}:$] vector of $x_{it}^{\omega_t}$ for all $i \in \mathcal{N}$ at stage $t \in \mathcal{T}$ in scenario $\omega_t \in \Omega_t$: $(x_{1,t}^{\omega_t}, x_{2,t}^{\omega_t}, ..., x_{n,t}^{\omega_t})$
	\item[$y_{it}:$] equals to 1 if CM is performed on component $i \in \mathcal{N}$ at stage $t \in \mathcal{T}$ and 0 otherwise
	\item[$y_{it}^{\omega_t}:$] $y_{it}$ in scenario $\omega_t \in \Omega_t$
	\item[$z_{t}:$] equals to 1 when any maintenance is performed at stage  $t \in \mathcal{T}$  and 0 otherwise
	\item[$z_{t}^{\omega_t}:$] $z_{t}$ in scenario $\omega_t \in \Omega_t$
	\item[$N_{0}:$] do-nothing set at the first stage, $N_{0}=\{i|x_{i,1}=0, i \in \mathcal{N}\}$
	\item[$N_{1}:$] maintenancce set at the first stage, $N_{1}=\{i|x_{i,1}=1, i \in \mathcal{N}\}$
\end{description}
\end{notation}

We consider condition-based maintenance optimization for multi-component systems. The system consists of multiple components with economic dependence. Significant cost savings can be achieved by maintaining multiple components jointly rather than separately.  We focus on systems with hidden failure which can only be revealed through inspection. For example, a production system may have failed but still operates, producing non-conforming products, and the failure can only be detected by inspection \cite{cassady2000combining}. We assume components deteriorate independently. Such an assumption is common for systems where components are not subject to common cause failures or the deterioration dependence among components is weak \cite{tian2011condition,dekker1996joint,wildeman1997dynamic,laggoune2009opportunistic}. Each component has $\{1,2,...,m-1,m\}$ condition states, where a larger state represents a worse yet functioning condition and state $m$ is the failure state.  All components are subject to stochastic degradation. Without maintenance intervention, the condition of a component cannot return to a better state. Inspection is performed periodically on the system to reveal the states of all components and each inspection is a decision stage. In some real-world problems, an inspection schedule is already in place based on experiences or required by regulations. For example, many refinery and chemical plants conduct annual or biannual turnarounds during which they inspect their equipment. In cases where the interval length needs to be determined, an optimal interval length can be determined using decomposition methods \cite{nicolai2008optimal,wildeman1997dynamic} or numerical search methods. At each decision stage, all failed components need to be correctively maintained and all functioning ones can be preventively maintained if desired. Both CM and PM restore a component to an as-good-as-new state, i.e., state 1.

This maintenance optimization problem is naturally a multi-stage stochastic integer program. At each stage $t \in \mathcal{T}$, we first observe all components' states $g_{it}$, $\forall i \in \mathcal{N}$. We then decide whether a component needs to be maintained ($x_{it},i \in \mathcal{N}, t \in \mathcal{T}$). All failed components $i$ are correctively maintained ($y_{it}=1,i \in \mathcal{N},t \in \mathcal{T}$). If there is any maintenance performed at stage $t \in \mathcal{T}$, the setup cost  is incurred ($z_t=1$).

We illustrate the decision process using a scenario tree in Figure \ref{Fig.scenTree}. In the scenario tree, we need to make maintenance decisions at each node $\omega_t \in \Omega_t,t \in \mathcal{T}$. Each node $\omega_t$ is characterized by a combination of all components' states, i.e., $(g_{1,t}^{\omega_t},g_{2,t}^{\omega_t},...,g_{n,t}^{\omega_t})$, and $\Omega_t$ is the set of all nodes at stage $t \in \mathcal{T}$.  For each node $\omega_t$, $t \in \mathcal{T}\backslash\{T\}$, it has a set of child nodes $\Omega(\omega_t)$ at stage $t+1$, where $\Omega(\omega_t)$ collects all possible combinations of all components' states. For each node $\omega_t$, $t \in \mathcal{T}\backslash\{1\}$, it has  a unique ancestor node $a(\omega_t)$ at stage $t-1$. A node path from the root node $(\omega_1)$ to a last stage node $(\omega_T \in \Omega_T)$ is referred to as a scenario. The total number of scenarios is $|\Omega_T|=m^{n(T-1)}$, which grows exponentially as the number of components and/or stages increase.

\begin{figure}[H] 
	\centering 
	\includegraphics[width=0.7\textwidth]{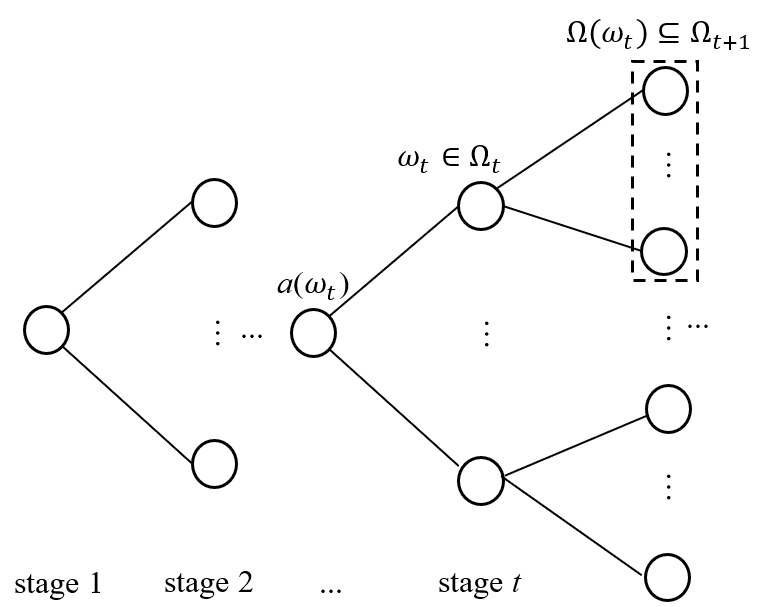} 
	\caption{Scenario tree} 
	\label{Fig.scenTree} 
\end{figure}

Our objective is to minimize the total cost over the planning horizon $\mathcal{T}$, where the total cost includes the first-stage cost and the expected second-stage cost of all nodes $\omega_2 \in \Omega_2$. For the cost at each node $\omega_t \in \Omega_t$ at stage $t \in \mathcal{T}\backslash\{T\}$, it consists of current-node cost and the expected cost of all child nodes $\Omega(\omega_t)$. For last stage nodes, i.e., $t=T$, the cost concerns the current-node cost only.

Given the component states $g_{it}^{\omega_t}$ for all components $i \in \mathcal{N}$ at node $\omega_t \in \Omega_t$ in stage $t \in \mathcal{T}\backslash\{T\}$, the probability from node $\omega_t$ to its child node $\omega_{t+1}$ depends on the maintenance decision $x_t^{\omega_t}$. For example, if the prior-maintenance state of a component in node $\omega_t$ is $g$, the post-maintenance transition probability is $Q(1,g^\prime)$ ($g^\prime \in G$) and $Q(g,g^\prime)$ ($g^\prime \geq g$) otherwise, and this leads to different node transition probabilities.

Denote  $p(\omega_{t+1}|x_t^{\omega_t})$ as the probability from node $\omega_t,t\in \mathcal{T}\backslash\{T\}$ to its child node $\omega_{t+1}$ given decision $x_t^{\omega_t}$, and  $Q_i(g,g^\prime)$ as the state transition probability from state $g$ to $g^\prime$ for component $i \in \mathcal{N}$. Since components deteriorate independently,  we have
$$
p(\omega_{t+1}|x_t^{\omega_t}) = \prod_{i \in \mathcal{N}}[Q_i(g_{it}^{\omega_t}, g_{i,t+1}^{\omega_{t+1}})(1-x_{it}^{\omega_t})+Q_i(1, g_{i,t+1}^{\omega_{t+1}})x_{it}^{\omega_t}].
$$
Next, we develop the multi-stage stochastic model:\\
Decision variables $(i \in \mathcal{N},\omega_t \in \Omega_t, t \in \mathcal{T})$:\\
$x_{it}^{\omega_t}$: 1 if component $i$ is maintained at node $\omega_t$ in  stage $t$, and 0 otherwise.\\
$y_{it}^{\omega_t}$: 1 if component $i$ is correctively maintained at node $\omega_t$ in  stage $t$, and 0 otherwise.\\
$z_{t}^{\omega_t}$: 1 if there is any maintenance occurs at node $\omega_t$ in  stage $t$, and 0 otherwise.\\
\noindent
\textbf{Multi-stage stochastic model (P1):}
\begin{spacing}{0.5}
\begin{equation}	
\label{obj1}
 V_1 = \min_{x,y,z} \sum_{i \in \mathcal{N}}c_{i,\text{pm}}x_{i,1}+\sum_{i \in \mathcal{N}}(c_{i,\text{cm}}-c_{i,\text{pm}})y_{i,1}+c_{\text{s}}z_{1}+\sum_{\omega_2 \in \Omega_2}p(\omega_2|x_{1})V_2(\omega_2)
\end{equation}
s.t.
\begin{equation}
\label{const1}
V_t(\omega_t)=
\left\{
\begin{array}{ll}
\begin{split}
\min_{x,y,z} &\sum_{i \in \mathcal{N}}c_{i,\text{pm}}x_{it}^{\omega_t}+\sum_{i \in \mathcal{N}}(c_{i,\text{cm}}-c_{i,\text{pm}})y_{it}^{\omega_t}\\
+&c_{\text{s}}z_{t}^{\omega_t}+\sum_{\omega_{t+1} \in \Omega(\omega_{t})}p(\omega_{t+1}|x_t^{\omega_t})V_{t+1}(\omega_{t+1})
\end{split}
& , i \in \mathcal{N},\omega_t \in \Omega_t, t \in \mathcal{T}\backslash\{T\},\\
\specialrule{0em}{1ex}{1ex}
\begin{split}
\min_{x,y,z} & \sum_{i \in \mathcal{N}}c_{i,\text{pm}}x_{it}^{\omega_t}+\sum_{i \in \mathcal{N}}(c_{i,\text{cm}}-c_{i,\text{pm}})y_{it}^{\omega_t}\\
+&c_{\text{s}}z_{t}^{\omega_t}
\end{split}
& , i \in \mathcal{N},\omega_t \in \Omega_t, t = T
\end{array}
\right.
\end{equation}
\begin{equation}	
\label{const2}
x_{it}^{\omega_t} \leq z_{t}^{\omega_t}, i \in \mathcal{N},\omega_t \in \Omega_t, t \in \mathcal{T}
\end{equation}
\begin{equation}	
\label{const3}
g_{it}^{\omega_t}(1-y_{it}^{\omega_t}) \leq m-1, i \in \mathcal{N},\omega_t \in \Omega_t, t \in \mathcal{T}
\end{equation}
\begin{equation}	
\label{const4}
y_{it}^{\omega_t} \leq x_{it}^{\omega_t}, i \in \mathcal{N},\omega_t \in \Omega_t, t \in \mathcal{T}
\end{equation}
\begin{equation}	
\label{const5}
x_{it}^{\omega_t},y_{it}^{\omega_t} \in \{0,1\}, i \in \mathcal{N},\omega_t \in \Omega_t, t \in \mathcal{T}
\end{equation}
\begin{equation}	
\label{const6}
z_{t}^{\omega_t} \in \{0,1\} ,\omega_t \in \Omega_t, t \in \mathcal{T}
\end{equation}
\end{spacing}

\vskip 0.5 cm

Objective function (\ref{obj1}) consists of the total cost at the first stage and the expected total cost at the second stage. The objective function $V_t(\omega_t)$ for node $\omega_t \in \Omega_t$ at stage $t \in \mathcal{T}$ is given by constraint (\ref{const1}). Constraints (\ref{const2}) ensure setup cost is incurred whenever a maintenance action is performed. Constraints (\ref{const3}) force CM actions on all failed components. Constraints (\ref{const4}) guarantee that the indicator of maintenance action $(x_{it}^{\omega_t})$ is set to 1 when CM is performed. Constraints (\ref{const5}) and (\ref{const6}) are integrality constraints for all decision variables.

As illustrated in Figure \ref{Fig.scenTree}, the problem size of P1  grows exponentially as the number of components increases. As discussed in the literature review, there is no general method to solve this problem due to the lack of structural properties in multi-stage stochastic integer programs with endogenous uncertainty. 
Therefore, we first consider a  two-stage problem.

The two-stage problem can be simplified by eliminating the second stage because the closed-form solutions for all second-stage subproblems can be obtained. Note that for the ease of notation, we drop the subscripts of $\omega_2$, $\Omega_2$ and $V_2$ in the two-stage model. First, for any subproblem $\omega \in \Omega$, the objective function $V(\omega)$ and all constraints are independent of the first-stage decisions and only depends on the components' states in scenario $\omega$. Because the second stage is the last stage of the two-stage problem, to minimize any subproblem, it is obvious that we only need to correctively maintain all failed components to satisfy constraints (\ref{const3}) and do nothing on functioning components.  Therefore, the optimal solutions in the second-stage subproblems are

\begin{spacing}{0.5}
	\begin{equation}
	\label{P2_opt1}
	(y_{i,2}^{\omega})^{\ast}=\lfloor{\frac{g_{i,2}^{\omega}}{m}\rfloor}, i \in \mathcal{N},\omega\in\Omega,
	\end{equation}
	\begin{equation}
	\label{P2_opt2}
	(x_{i,2}^{\omega})^{\ast}=(y_{i,2}^{\omega})^{\ast}, i \in \mathcal{N},\omega\in\Omega
	\end{equation}
	and 
	\begin{equation}
	\label{P2_opt3}
	(z_{2}^{\omega})^{\ast}=\min(1,\sum_{i \in \mathcal{N}}(x_{i,2}^{\omega})^{\ast}),\omega\in\Omega.
	\end{equation}
\end{spacing}
\vskip 0.5 cm

Based on Equations (\ref{P2_opt1}) to (\ref{P2_opt3}), the two-stage model (P2) is described as follows:

\noindent
\textbf{Two-stage stochastic model (P2):}
\begin{spacing}{0.5}
\begin{equation}	
\label{P2O1}
\text{min} \sum_{i \in \mathcal{N}}c_{i,\text{pm}}x_{i,1}+\sum_{i \in \mathcal{N}}(c_{i,\text{cm}}-c_{i,\text{pm}})y_{i,1}+c_{\text{s}}z_{1}+V
\end{equation}

s.t.
\begin{equation}	
\label{P2C1_1}
x_{i,1} \leq z_{1}, i \in \mathcal{N}
\end{equation}
\begin{equation}	
\label{P2C1_2}
g_{i,1}(1-y_{i,1}) \leq m-1, i \in \mathcal{N}
\end{equation}
\begin{equation}	
\label{P2C1_3}
y_{i,1} \leq x_{i,1}, i \in \mathcal{N}
\end{equation}
\begin{equation}	
\label{P2C1_4}
x_{i,1},y_{i,1} \in \{0,1\}, i \in \mathcal{N}
\end{equation}
\begin{equation}	
\label{P2C1_5}
z_{1} \in \{0,1\}
\end{equation}
\end{spacing}
\begin{spacing}{1}
where 
\begin{equation}
\label{P2V}
\begin{split}
V=\sum_{\omega \in \Omega}p(\omega|x_{1})V(\omega)=&\sum_{i \in \mathcal{N}}(Q_{i}(g_{i,1},m)(1-x_{i,1})+
Q_{i}(1,m)x_{i,1}))c_{i,\text{cm}} \\
&+(1-\prod_{i \in \mathcal{N}}(1-Q_{i}(g_{i,1},m)(1-x_{i,1})-
Q_{i}(1,m)x_{i,1}))c_{\text{s}}.
\end{split}
\end{equation}
\end{spacing}
The two-stage model can be directly used in mission-critical applications where successfully completing a mission (one-period) is the primary concern.

\section{Structural properties of the two-stage model}
\label{secProp}
In this section, we establish three structural properties for P2. The first property provides an optimal solution to P2 based on the optimal solution without considering economic dependence. Because the optimal solution without considering economic dependence can be obtained easily, we can quickly identify the optimal solution to P2 when the condition in Proposition \ref{prop1} is satisfied. The second property establishes the condition when changing the decision(s) of certain component(s) from do-nothing to PM reduces the total maintenance cost.
The third property establishes the condition when changing the decision(s) of certain component(s) from  PM to do-nothing reduces the total maintenance cost. Propositions 2 and 3 are the theoretical foundation of Algorithm 1 that solves P2 optimally.

\begin{proposition}
\label{prop1}
If $\tilde{x}_{i,1}^\ast=1$ $\forall i \in \mathcal{N}$, then $x_{i,1}^\ast=1$ $\forall i \in \mathcal{N}$. $\qed$
\end{proposition}
\begin{proof}
	See Appendix A.1.
\end{proof}

Proposition \ref{prop1} shows that it is optimal to maintain all components (i.e. $x_{i,1}^\ast=1$) if all components need to be maintained when ignoring economic dependence (i.e. $\tilde{x}_{i,1}^\ast=1$). The optimal maintenance decision without considering economic dependence $\tilde{x}_{i,1}^\ast$ of component  $i \in \mathcal{N}$ can be obtained easily as follows:
\begin{equation}
\tilde{x}_{i,1}^\ast=
\left\{
\begin{array}{ll}
1, & \text{when} \; Q_i(g_{i,1},m) > \dfrac{c_{i,\text{pm}+c_\text{s}}}{c_{i,\text{cm}+c_\text{s}}}+Q_i(1,m)\; \text{or} \; g_{i,1}=m,\\
\specialrule{0em}{1ex}{1ex}
0, & \text{otherwise}.
\end{array}
\right.
\end{equation}

Proposition \ref{prop1} leads to an optimal solution to P2 when $\tilde{x}_{i,1}^\ast=1$ for all components $i \in \mathcal{N}$. However, the condition of $\tilde{x}_{i,1}^\ast=1$ for all $i \in \mathcal{N}$ is a special scenario. Next, we explore more general structural properties of the two-stage model.

\begin{definition}
A partition $(N_0,N_1)$ of set $\mathcal{N}$, i.e., $N_0 \cup N_1=\mathcal{N}$ and $N_0 \cap N_1 = \emptyset$, is a solution to P2, where $N_0$ is the do-nothing set that collects all components that are not maintained at the first stage, i.e., $N_0 = \{i|x_{i,1} = 0, i \in \mathcal{N}\}$ and $N_1$ is the maintenancce set that includes all components that are maintained at the first stage, i.e., $N_1 = \{i|x_{i,1} = 1,i \in \mathcal{N}\}$ . $\qed$
\end{definition}

A partition $(N_0,N_1)$ of $\mathcal{N}$ is \textit{feasible} if every failed component at the first stage belongs to $N_1$. Therefore, determining the optimal $x_{1}^\ast$ is now equivalent to find out a feasible and optimal partition $(N_0^\ast, N_1^\ast)$ of $\mathcal{N}$ that minimizes the total cost. Next, we give two propositions regarding how to improve a feasible partition $(N_0,N_1)$.

\begin{proposition}
\label{prop2}
Consider two feasible partitions  $(N_0,N_1)$ and $(N_0^\prime, N_1^\prime)$ of $\mathcal{N}$. Let  $C$ and $C^\prime$ be their respective total costs. If $N_1^\prime \backslash  N_1 = N$ and $N \neq \emptyset$, we have $C^\prime<C$ if and only if $\Delta_{\text{r}}(N_0, N_1, N)<1$, where
\begin{spacing}{0.5}
\begin{displaymath}
\Delta_{\text{r}}(N_0, N_1, N)=
\left\{
\begin{array}{ll}
\dfrac{\sum_{k \in N}\rho_k}{r_{N}}\times\dfrac{1}{p(N_0,N_1)}, & {\rm when} \; N_1 \neq \emptyset,\\
\specialrule{0em}{1ex}{1ex}
\dfrac{1+\sum_{k \in N}\rho_k}{r_{N}}\times\dfrac{1}{p(N_0,N_1)}, & {\rm otherwise}.
\end{array}
\right.
\end{displaymath} 
$$
r_N = \prod_{i \in N}\dfrac{1-Q_i(1,m)}{1-Q_i(g_{i,1},m)}-1,
$$ 
$$
\rho_k = \dfrac{c_{k,\text{pm}}-(Q_k(g_{k,1},m)-Q_k(1,m))c_{k,\text{cm}}}{c_\text{s}},
$$
and
$$
p(N_0,N_1) = \prod_{i \in N_0}(1-Q_i(g_{i,1},m))\prod_{i \in N_1}(1-Q_i(1,m))\qed
$$
\end{spacing}
\end{proposition}

\begin{proof}
	See Appendix A.2.	
\end{proof}

Proposition \ref{prop2} helps to quickly identify a  set $N \subseteq N_0$ to improve the current partition $(N_0, N_1)$ by moving set $N$ from the do-nothing set to the maintenance set. $\Delta_{\text{r}}(N_0, N_1, N)$ consists of two parts: $\dfrac{\sum_{k \in N}\rho_k}{r_{N}}$ and  $p(N_0,N_1)$. The first part is determined by the components in set $N$ and the second part is the probability that all components will survive in the second stage given the current decision partition, i.e., $(N_0,N_1)$. The probability $p(N_0,N_1)$ increases as more components are maintained.

Let us first examine the condition in Proposition \ref{prop2} when $|N|=1$. Suppose $N=\{k\}$, $k \in N_0$, Proposition \ref{prop2}  provides the condition of improving the current partition by maintaining component $k$. When  $N_1 \neq \emptyset$, we have 
\begin{equation}
\label{EquationP2}
\begin{split}
\Delta_{\text{r}}(N_0,N_1,\{k\})=&\dfrac{(c_{k,\text{pm}}-(Q_k(g_{k,1},m)-Q_k(1,m))c_{k,\text{cm}})(1-Q_k(g_{k,1},m))}{c_{\text{s}}(Q_k(g_{k,1},m)-Q_k(1,m))p(N_0,N_1)}\\
\approx&\dfrac{(c_{k,\text{pm}}-Q_k(g_{k,1},m)c_{k,\text{cm}})(1-Q_k(g_{k,1},m))}{c_{\text{s}}Q_k(g_{k,1},m)p(N_0,N_1)}
\end{split}
\end{equation}
because $Q_k(1,m)$ is the state transition probability from the perfect state to the failed state and therefore $Q_k(1,m) \approx 0$ holds in many scenarios (e.g., when inspection interval is not long). Based on Equation (\ref{EquationP2}), we have several important observations: (1) $\Delta_{\text{r}}(N_0,N_1,\{k\})$ increases as $c_{k,{\rm pm}}$ increases. The increase in $\Delta_{\text{r}}$ indicates that it is less likely that we change the decision on $k$ from no maintenance to PM. This is because a higher PM cost makes it less cost-effective to perform PM at the first stage. (2)  $\Delta_{\text{r}}(N_0,N_1,\{k\})$ increases as $c_{k,\text{cm}}$ decreases. It is less incentive to perform PM on component $k$ with other components at the first stage when CM cost of component $k$ is lower. (3)  $\Delta_{\text{r}}(N_0,N_1,\{k\})$ increases as $c_{\text{s}}$ decreases, because the decrease of setup cost makes sharing setup cost at the first stage less cost-effective. And (4)  $\Delta_{\text{r}}(N_0,N_1,\{k\})$ increases as $g_{k,1}$ decreases, meaning a better component condition makes it less worthy to maintain the component at the first stage. Similar patterns can be observed when $|N|\geq 2$. Proposition \ref{prop2} also provides the maintenance action for a new component. Specifically, when $g_{k,1}=1$, $k \in N_0$, we have $\Delta_{\text{r}}(N_0,N_1,\{k\}) = +\infty > 1$, which implies we should never maintain a new component.

\begin{proposition}
	\label{prop3}
	Consider two feasible partitions  $(N_0,N_1)$ and $(N_0^\prime, N_1^\prime)$ of $\mathcal{N}$. Let  $C$ and $C^\prime$ be their respective total costs. If $N_0^\prime \backslash  N_0 = N$ and $N \neq \emptyset$, we have $C^\prime<C$ if and only if $\Delta_{\text{s}}(N_0, N_1, N)>1$, where
	\begin{displaymath}
	\Delta_{\text{s}}(N_0, N_1, N)=
	\left\{
	\begin{array}{ll}
	\dfrac{\sum_{k \in N}\rho_k}{s_{N}}\times\dfrac{1}{p(N_0,N_1)}, & {\rm when} \; N_1 \neq \emptyset,\\
	\specialrule{0em}{1ex}{1ex}
	\dfrac{1+\sum_{k \in N}\rho_k}{s_{N}}\times\dfrac{1}{p(N_0,N_1)}, & {\rm otherwise}.
	\end{array}
	\right.
	\end{displaymath} 
	$$
	s_N = 1-\prod_{i \in N}\dfrac{1-Q_i(g_{i,1},m)}{1-Q_i(1,m)},
	$$ 
	and values $\rho_k$ and $p(N_0,N_1)$ are the defined in Proposition \ref{prop2}. $\qed$
\end{proposition}

\begin{proof}
	See Appendix A.3.	
\end{proof}

Proposition \ref{prop3} helps to quickly identify a set $N \subseteq N_1$ to improve the current partition $(N_0, N_1)$ by moving set $N$ from the maintenance set to the do-nothing set. Note that in contrast to considering  $N \subseteq N_0$ in Proposition \ref{prop2}, Proposition \ref{prop3} considers $N \subseteq N_1$.

We similarly first investigate the condition in Proposition \ref{prop3} when $|N|=1$. Suppose $|N|=\{k\}$, $k \in N_1$, Proposition \ref{prop3} establishes the condition of improving the current partition by not maintaining component $k$. When  $N_1^\prime \neq \emptyset$, we have 
\begin{equation}
\label{EquationP3}
\begin{split}
\Delta_{\text{s}}(N_0,N_1,\{k\})=&\dfrac{(c_{k,\text{pm}}-(Q_k(g_{k,1},m)-Q_k(1,m))c_{k,\text{cm}})(1-Q_k(1,m))}{c_{\text{s}}(Q_k(g_{k,1},m)-Q_k(1,m))p(N_0,N_1)}\\
\approx&\dfrac{c_{k,\text{pm}}-Q_k(g_{k,1},m)c_{k,\text{cm}}}{c_{\text{s}}Q_k(g_{k,1},m)p(N_0,N_1)}
\end{split}
\end{equation}
because $Q_k(1,m) \approx 0$ in many scenarios. Examining Equation (\ref{EquationP3}), we observe similar patterns regarding whether changing a component from PM to do-nothing reduces the total maintenance costs  as the ones we see from Equation (\ref{EquationP2}).

\begin{corollary}
\label{cor1}
Let $(N_0 \cup N_{\rm u},N_1)$ and $(N_0,N_1 \cup N_{\rm u})$ be two feasible partitions of $\mathcal{N}$. For any set $N \subseteq N_{\rm u}$ and $N \neq \emptyset$, we have $\Delta_{\text{r}}(N_0 \cup N_{\rm u}, N_1, N) \geq \Delta_{\text{s}}(N_0, N_1 \cup N_{\rm u}, N)$, where  equality holds when  $N = N_{\rm u}$. $\qed$
\end{corollary}

\begin{proof}
	See Appendix B.1.
\end{proof}
Corollary \ref{cor1} shows that any set $N \subset N_{\rm u}$ satisfies either  Proposition \ref{prop2} or Proposition \ref{prop3} or none. When $N=N_{\rm u}$, set $N$ satisfies either Proposition 2 or Proposition 3. This corollary is needed to prove Proposition \ref{prop4} in the next section. 

\section{Solution algorithms}
\label{secAlg}
Based on Propositions \ref{prop2} and \ref{prop3}, we  design  Algorithm 1 that finds the optimal partition $(N_0^\ast,N_1^\ast)$ for P2. Although the computational studies in the next section show that Algorithm 1 is fast for most test cases, the time complexity of Algorithm 1 is $O(2^n)$ in the worst case scenario. Therefore, we develop Algorithm 2  to heuristically search a better solution based on the results from the early termination of Algorithm 1. We further use the two-stage model and rolling horizon technique to approximate the multi-stage problem P1.

\subsection{Algorithm 1}

Propositions \ref{prop2} and \ref{prop3} help to find a better solution given any feasible solution. However, they do not necessarily lead any feasible solution to an optimal one. Next we show that if Propositions \ref{prop2} and \ref{prop3} are applied following a certain procedure, an optimal solution can be obtained.

Let $N_{\rm u}$ be the undetermined set in which all components' first-stage decisions are not determined. Constructing an optimal partition $(N_0,N_1)$ implies optimally moving all subsets $N \subseteq N_{\rm u}$ to $N_0$ or $N_1$. The proposed procedure starts from searching all subsets with $|N|=1$ and increases the cardinality of $N$ by 1 until some $N$ is moved to $N_1$ based on Proposition \ref{prop2} or $N_0$ based on Proposition \ref{prop3}. Because the conditions of Propositions \ref{prop2} and \ref{prop3} change after moving $N$, the search restarts from $|N|=1$. This process is repeated until $N_{\rm u}=\emptyset$. It can be easily verified that $2^n$ sets need to be examined in the worst case scenario.

\renewcommand{\baselinestretch}{1} \normalsize
\begin{algorithm}[h]
	\caption{Determining an optimal partition $(N_0^\ast,N_1^\ast)$ for P2}
	\begin{algorithmic}[1]
		\Require
		Component set $\mathcal{N}$;\par
		PM cost $c_{i,\text{pm}}$ and CM cost $c_{i,\text{cm}}$ $\forall i \in \mathcal{N}$ and setup cost $c_\text{s}$;\par
		Component state $g_{i,1}$ $\forall i \in \mathcal{N}$\par
		State transition probability $Q_i(1,m)$ and $Q_i(g_{i,1},m)$ $\forall i \in \mathcal{N}$.
		\Ensure
		Optimal partition $(N_0^\ast,N_1^\ast)$ of $\mathcal{N}$
		\State 	Initial  $N_0^\ast \gets \emptyset$, $N_1^\ast \gets \emptyset$, $j \gets 1$, undetermined set $N_{\rm u} \gets \mathcal{N}$;
		\For{$i \in \mathcal{N}$}// \textit{maintain all failed components}
		\If {$g_{i,1} = m$}
		\State $N_1^\ast = N_1^\ast\cup \{i\}$,$N_{\rm u} = N_{\rm u} \backslash \{i\}$;
		\EndIf
		\EndFor
		\While {$N_{\rm u} \neq \emptyset$} // \textit{Keep moving $N \subseteq N_{\rm u}$ to $N_0^\ast$ and $N_1^\ast$ until $N_{\rm u} = \emptyset$}
		\State  $N_0 \gets \emptyset$, $N_1 \gets \emptyset$, $u \gets |N_{\rm u}|$;
		\State $N^j \gets$ all subsets of $N_{\rm u}$ with cardinality $j$;
		\For {each set $N \in N^j$}
		\If {$\Delta_{\text{r}}(N_0^\ast \cup N_{\rm u}, N_1^\ast, N)<1$} 
		\State $N_1 \gets N_1 \cup  N$; //\textit{$N$ satisfies Proposition \ref{prop2}, move $N$ to $N_1$}
		\ElsIf {$\Delta_{\text{s}}(N_0^\ast, N_1^\ast \cup N_{\rm u}, N) \geq 1$}
		\State $N_0 \gets N_0 \cup  N$;  //\textit{$N$ satisfies Proposition \ref{prop3}, move $N$ to $N_0$}
		\EndIf
		\EndFor	
		\State 	$N_{\rm u} \gets N_{\rm u}\backslash (N_1 \cup N_0)$, $N_0^\ast \gets N_0^\ast \cup N_0$, $N_1^\ast \gets N_1^\ast \cup N_1$; 
		\If {$u > |N_{\rm u}|$} //\textit{ If $N_{\rm u}$ is reduced }
		\State $j \gets 1$; //\textit{Search $N$ from $|N|=1$.}
		\Else
		\State $j \gets j+1$;//\textit{Search $N$ at a higher cardinality. }
		\EndIf
		\EndWhile\\
		\Return $(N_0^\ast,N_1^\ast)$;
	\end{algorithmic}
\end{algorithm}
\renewcommand{\baselinestretch}{1.5} \normalsize

Specifically, we initialize $N_0=\emptyset$, $N_1$ to include all failed components to ensure the feasibility, and $N_{\rm u}=\mathcal{N}\backslash N_1$. If there is any subset $N \subseteq N_{\rm u}$ with $|N|=1$ satisfies Proposition \ref{prop2} (Proposition \ref{prop3}), we move $N$ from $N_{\rm u}$ to $N_1$  ($N_0$) after all subsets with $|N|=1$ are searched. If there is no subset $N$ with $|N|=1$ satisfies Propositions \ref{prop2} or \ref{prop3}, we search the subsets $N$ with $|N|=2,3,...,|N_{\rm u}|$ in an ascending order until some $N$ satisfying the condition in Propositions \ref{prop2} or \ref{prop3} is obtained and moved out of $N_{\rm u}$. We then update $N_{\rm u}$ and restart to search the subset $N$ from $|N|=1$. The construction of the optimal partition $(N_0,N_1)$ terminates when $N_{\rm u}=\emptyset$. The optimality of partition $(N_0,N_1)$ given by Algorithm 1 is proved in Proposition \ref{prop4}.

\begin{proposition}
	\label{prop4}
	The partition $(N_0^\ast,N_1^\ast)$ given by Algorithm 1 is an optimal partition. $\qed$
\end{proposition}
\begin{proof}
	See Appendix A.4.	
\end{proof}

\subsection{Algorithm 2}
As stated previously, Algorithm 1 requires to examine $2^n$ sets in the worst case scenario. To ensure that we obtain a high-quality solution in a reasonable amount of time, we develop Algorithm 2 that heuristically finds a sub-optimal solution based on Algorithm 1. 

\renewcommand{\baselinestretch}{1} \normalsize
\begin{algorithm}[]
	\caption{Heuristic algorithm for P2}
	\begin{algorithmic}[1]
		\Require
		$N_0^\ast$ and $N_1^\ast$ $\gets$ results from Algorithm 1 by setting $J$ as the maximum set cardinality to search;\par
		Maximum number of partitions $M$; \par
		\Ensure
		Sub-optimal partition $(N_0^\prime,N_1^\prime)$ of $\mathcal{N}$.
		\State $N_{\rm u} \gets \mathcal{N} \backslash (N_0^\ast \cup N_1^\ast)$;
		\State Randomly generate $M$ partitions $(N_0^\ast\cup \bar{N}_0, N_1^\ast\cup \bar{N}_1)$, including $(N_0^\ast\cup N_{\rm u},N_1^\ast)$ and $(N_0^\ast, N_1^\ast\cup N_{\rm u})$, where  $\bar{N}_0 \cup \bar{N}_1 = N_{\rm u}$;
		\State Select the best partition $(N_0^\prime,N_1^\prime)$ that has the minimum cost among $M$ partitions;\\
		\Return $(N_0^\prime,N_1^\prime)$;
	\end{algorithmic}
\end{algorithm}
\renewcommand{\baselinestretch}{1.5} \normalsize
Specifically, we first terminate Algorithm 1 after the cardinality of $|N|$ exceeds the maximum cardinality $J$ specified, which means we only search the component set that has no more than $J$ components. Based on  $N_0^\ast$ and $N_1^\ast$ obtained from early termination of Algorithm 1, we randomly generate $M$ partitions $(\bar{N_0},\bar{N_1})$ of undetermined set $N_{\rm u} = \mathcal{N} \backslash (N_0^\ast \cup N_1^\ast)$, and select the best partition $(N_0^\ast \cup \bar{N_0},N_1^\ast \cup \bar{N_1})$ of $\mathcal{N}$. Note that it is suggested to include the options of maintaining all and none of undetermined components as candidate solutions, because many of our experiments show that it is likely that the optimal partition is either $(N_0^\ast\cup N_{\rm u},N_1^\ast)$ or $(N_0^\ast,N_1^\ast\cup N_{\rm u})$. 

\subsection{Algorithm 3}
We further use P2 to approximate the multi-stage model (P1) by utilizing the rolling horizon technique. At each decision period $t$, we solve P2 which consists of periods $t$ and $t + 1$ using Algorithm 2 and employ the first-stage solutions as the decisions for period $t$. The states components transition to in the next period (i.e., period $t + 1$) are determined by the decision made in the previous period (i.e., period $t$) and the transition probabilities. We then solve a new P2 consisting of decision periods $t +1$ and $t + 2$ and use the first-stage solution as the decisions for period $t + 1$. This process is repeated until the last decision period is reached. This procedure is summarized in Algorithm 3.

\renewcommand{\baselinestretch}{1} \normalsize
\begin{algorithm}[]
	\caption{Solving multi-stage model  using a rolling-horizon approach}
	\begin{algorithmic}[1] 
		\Ensure
		A solution over planning horizon $\mathcal{T}$
		\For {$t \in \mathcal{T}\backslash\{T\}$}
		\State Solve P2 using Algorithm 2, and obtain the first-stage solutions $x_{it},y_{it}$ and $z_t$ for all components $i \in \mathcal{N}$;
		\EndFor
		\State The last stage solutions $x_{iT},y_{iT}$ and $z_T$ for all components $i \in \mathcal{N}$ are given by Equations (\ref{P2_opt1}) to (\ref{P2_opt3});\\
		\Return $x_{it},y_{it}$ and $z_t$ for all components $i \in \mathcal{N}$ at all stages $t \in \mathcal{T}$.
	\end{algorithmic}
\end{algorithm}
\renewcommand{\baselinestretch}{1.5} \normalsize

In many real applications, when more degradation information becomes available upon inspection at each decision period, Bayesian updating can be easily performed to obtain more accurate degradation distributions and better-informed maintenance decisions can be made.

\section{Computational study}\label{sectionCom}
In this section, we first linearize P2 so that  small-scale problems of P2 can be solved by commercial solvers such as CPLEX for comparison purposes. We then conduct computational studies to examine the performance of Algorithms 1 and 2. The proposed models and algorithms are then illustrated by two real-world cases.

\subsection{Linearization of P2}\label{sectionLin}
In Equation (\ref{P2V}), the term $\prod_{i \in \mathcal{N}}(1-Q_{i}(g_{i,1},m)(1-x_{i,1})-Q_{i}(1,m)x_{i,1})$ is non-linear, which is linearized first. The term $\prod_{i \in \mathcal{N}}(1-Q_{i}(g_{i,1},m)(1-x_{i,1})-Q_{i}(1,m)x_{i,1})$ can be expanded to a polynomial function of $x_{i,1},i \in \mathcal{N}$, with degree of $n$. After the expansion, we observe that all non-linear terms are  the products of multiple (from 2 to $n$) binary decision variables $x_{i,1},i \in \mathcal{N}$. Standard linearization method for the multiplication of multiple binary variables are applied here \cite{rardin1998optimization}. After linearization, we replace Equation (\ref{P2V}) by Equation (\ref{P2VL}) in model P2,

\begin{equation}
\label{P2VL}
\begin{split}
	V=&\sum_{i \in \mathcal{N}}(Q_{i}(g_{i,1},m)(1-x_{i,1})+
	Q_{i}(1,m)x_{i,1}))c_{i,\text{cm}} \\
	&+(1-\sum_{j=0}^{1}\sum_{k=1}^{|N^j|}\prod_{i \in N_k^j}\prod_{r \in \mathcal{N}\backslash N_k^j}a_ix_{i,1}b_r-
	\sum_{j=2}^{n}\sum_{k=1}^{|N^j|}u_k^j\prod_{i \in N_k^j}\prod_{r \in \mathcal{N}\backslash N_k^j}a_ib_r)c_{\text{s}}
\end{split}
\end{equation}
where
\begin{spacing}{0.5}
\begin{equation}
\label{P2LC1}
b_{i}=1-Q_i(g_{i,1},m), i \in \mathcal{N}
\end{equation}
\begin{equation}
\label{P2LC2}
a_{i}=Q_i(g_{i,1},m)-Q_i(1,m),i \in \mathcal{N}
\end{equation}
\begin{equation}
\label{P2LC3}
u_j^k \leq x_{i,1}, j \in \{2,...,n\}, k \in \{1,...,|N^j|\}, i \in N_k^j
\end{equation}
\begin{equation}
\label{P2LC4}
u_j^k \geq \sum_{i \in N_k^j}x_{i,1}-(j-1), j \in \{2,...,n\}, k \in \{1,...,|N^j|\}
\end{equation}  
\begin{equation}
\label{P2LC5}
u_j^k \in \{0,1\}, j \in \{2,...,n\}, k \in \{1,...,|N^j|\}.
\end{equation}  
\end{spacing}
\vskip 0.5 cm

Note that set $N^j$ collects all subsets of $\mathcal{N}$ that have cardinality $j$ and therefore  $|N^j|=\dbinom{n}{j}$, $j \in \{0,1,2,...,n\}$.  For each set $N_k^j \in N^j, k \in \{1,...,|N^j|\}$, we have $N_k^j \subseteq \mathcal{N}$ and $|N_k^j|=j$.

\subsection{Computational studies}
We first compare the computational time of Algorithm 1 with CPLEX for small-scale problems. We then examine the computational time and cost error of Algorithms 1 and 2 for large-scale problems.

We assume the degradation of all components can be described by gamma processes with shape parameter $\alpha t$ and rate parameter $\gamma$. The continuous degradation levels are divided into several intervals to represent different states, and the transition probabilities can be computed accordingly. Without loss of generality, we assume inspection interval is 1. We arbitrarily set $M=100$, which is the maximum partitions generated in Algorithm 2. We consider systems with different number of components $n \in \{10,11,...,19\}$. For each $n$, we consider 100 instances with different combinations of degradation processes, and costs of PM and CM. The degradation parameters, and the costs of PM and CM are drawn from uniform distributions $U(\cdot,\cdot)$. Therefore, a total of 10,000 experiments are run. Table   \ref{TableNumBaseLine} summarizes the baseline parameters.

\begin{table}[h]  
	\centering  
	\caption{Baseline parameters for numerical example} 
	\label{TableNumBaseLine}
	\begin{tabular}{@{}cccccccccc@{}}
	\toprule
	$\alpha$   & $\gamma$   & \begin{tabular}[c]{@{}c@{}}failure\\ threshold\end{tabular} & $m$   & \begin{tabular}[c]{@{}c@{}}PM\\ cost\end{tabular} & \begin{tabular}[c]{@{}c@{}}CM\\ cost\end{tabular} & $c_\text{s}$  & instance & $M$   \\ \midrule
	$U(1,5)$ & $U(0.2,1)$ & 20  & 11 & $U(1,5)$  &$U(10,30) $            & 20 & 100                                                   & 100   \\ \bottomrule
\end{tabular}
\end{table}

For each $n$, we examine the average performance of 100 problem instances. Table \ref{TableNumSmall} presents the computational times of solving P2 by CPLEX and Algorithm 1 for different numbers of components. NA is reported when the computational time is either longer than 1 day or out of memory. From Table \ref{TableNumSmall}, we can see that the computational time of using CPLEX grows exponentially as the number of components increases. In contrast, Algorithm 1 finds the optimal solutions in a short amount of time.

\begin{table}[h]
\centering  
\caption{Computational time of solving P2 (in seconds)} 
\label{TableNumSmall}
\begin{tabular}{ccc|ccc}
	\hline
	$n$  & solver & Algorithm 1 & $n$  & solver  & Algorithm 1 \\ \hline
	10 & 0.422  & 0.0002      & 15 & 186.171 & 0.0004      \\
	11 & 1.029  & 0.0003      & 16 & 832.674 & 0.0004      \\
	12 & 3.064  & 0.0003      & 17 & 5093.700      &     0.0005\\
	13 & 11.943 & 0.0003      & 18 &  NA       &    0.0005         \\
	14 & 46.014 & 0.0004      & 19 &  NA       &    0.0005         \\ \hline
\end{tabular}
\end{table}

Next, we  investigate the performances of Algorithms 1 and 2 for large-scale problems. For each $n$, we similarly examine 100 problem instances. Note that CPLEX cannot solve any large-scale cases tested. Table \ref{TableNumLarge1} summarizes the performance of Algorithms 1 and 2 for large-scale problems. For each $n$ in Algorithm 1, we are interested in the average computational time of the 100 problem instances (avg. time), the maximum computational time (max time), the average $j_{\rm m}$ (avg. $j_{\rm m}$) and the maximum $j_{\rm m}$ (max $j_{\rm m}$), where $j_{\rm m}$ is the maximum set cardinality $j$ that Algorithm 1 searched. From Table \ref{TableNumLarge1}, we can see that the average time in general increases as the number of components increases. It is also noted that the maximum search time of Algorithm 1 increases substantially as the number of components increases.  This is because the solution space increases significantly as the number of components increases. As a result, Algorithm 1 may have to search more sets at higher cardinalities of $|N|$ before reaching the optimality criterion (i.e., undetermined set is empty). This is evidenced by the increase of $j_{\rm m}$. A higher cardinality generates more sets to be examined in Propositions \ref{prop2} and \ref{prop3}, and this consumes more computational time. For Algorithm 2, we examine the computational time and cost error for different stopping criteria $J$, which is the maximum set cardinality that Algorithm 1 allowed to search. We note that cost errors are all zero compared with the true objective value obtained by Algorithm 1, which shows Algorithm 2 can find high-quality solutions within a reasonable amount of time. We similarly observe that computational time increases as $J$ increases.

\begin{sidewaystable}[]
\centering
\caption{Performance of Algorithm 1$\&$2 in large-scale problems (in seconds)}
\label{TableNumLarge1}
\begin{tabular}{ccccc|cccccccccccc}
\hline
\multirow{3}{*}{$n$} & \multicolumn{4}{c|}{Algorithm 1} & \multicolumn{12}{c}{Algorithm 2} \\ \cline{2-17} 
& \multirow{2}{*}{\begin{tabular}[c]{@{}c@{}}avg.\\ time\end{tabular}} & \multirow{2}{*}{\begin{tabular}[c]{@{}c@{}}max\\ time\end{tabular}} & \multirow{2}{*}{\begin{tabular}[c]{@{}c@{}}avg.\\ $j_{\rm m}$ \end{tabular}} & \multirow{2}{*}{\begin{tabular}[c]{@{}c@{}}max\\ $j_{\rm m}$\end{tabular}} & \multicolumn{6}{c}{avg. time} & \multicolumn{6}{c}{max time} \\ \cline{6-17} 
&  &  &  &  & $J$=1 & $J$=2 & $J$=3 & $J$=4 & $J$=5 & \multicolumn{1}{c|}{$J$=6} & $J$=1 & $J$=2 & $J$=3 & $J$=4 & $J$=5 & $J$=6 \\ \hline
20 & 0.001 & 0.002 & 1.19 & 4 & 0.001 & 0.001 & 0.001 & 0.001 & 0.001 & \multicolumn{1}{c|}{0.001} & 0.003 & 0.003 & 0.004 & 0.001 & 0.001 & 0.001 \\
40 & 0.002 & 0.006 & 1.01 & 2 & 0.002 & 0.002 & 0.002 & 0.002 & 0.002 & \multicolumn{1}{c|}{0.002} & 0.008 & 0.004 & 0.004 & 0.003 & 0.006 & 0.006 \\
60 & 0.004 & 0.012 & 1.00 & 1 & 0.004 & 0.004 & 0.004 & 0.004 & 0.004 & \multicolumn{1}{c|}{0.004} & 0.012 & 0.012 & 0.010 & 0.011 & 0.012 & 0.011 \\
80 & 0.007 & 0.009 & 1.00 & 1 & 0.007 & 0.007 & 0.007 & 0.007 & 0.007 & \multicolumn{1}{c|}{0.007} & 0.012 & 0.021 & 0.022 & 0.018 & 0.021 & 0.017 \\
100 & 0.015 & 0.236 & 1.08 & 3 & 0.013 & 0.013 & 0.015 & 0.014 & 0.014 & \multicolumn{1}{c|}{0.015} & 0.036 & 0.062 & 0.251 & 0.239 & 0.239 & 0.240 \\
120 & 0.023 & 0.247 & 1.11 & 3 & 0.019 & 0.020 & 0.022 & 0.022 & 0.023 & \multicolumn{1}{c|}{0.022} & 0.073 & 0.073 & 0.278 & 0.246 & 0.253 & 0.248 \\
140 & 11.4 & 939 & 1.65 & 7 & 0.034 & 0.044 & 0.119 & 0.405 & 1.284 & \multicolumn{1}{c|}{4.307} & 0.078 & 0.146 & 0.948 & 7.750 & 45.33 & 230.5 \\
160 & 7.07 & 381 & 1.86 & 6 & 0.052 & 0.066 & 0.160 & 0.589 & 1.779 & \multicolumn{1}{c|}{7.006} & 0.123 & 0.218 & 1.307 & 10.53 & 71.90 & 380.9 \\
180 & 9.85 & 348 & 2.19 & 6 & 0.072 & 0.102 & 0.355 & 1.447 & 4.602 & \multicolumn{1}{c|}{9.940} & 0.174 & 0.396 & 3.022 & 18.51 & 137.9 & 349.1 \\
200 & 106 & 6379 & 2.65 & 8 & 0.091 & 0.134 & 0.515 & 2.559 & 9.622 & \multicolumn{1}{c|}{24.75} & 0.173 & 0.317 & 2.494 & 23.79 & 185.4 & 1205 \\ \hline
\end{tabular}
\end{sidewaystable}

\newpage

We further examine the performance the two-stage rolling horizon approach (Algorithm 3) in approximating the multi-stage model by comparing results from our approach with optimal results on small-scale problems where exact solutions can be obtained. We arbitrarily use one problem setting by taking a sample of parameters based on Table \ref{TableNumBaseLine} and replicate the problem 1000 times to obtain the average cost using Algorithm 3. Enumeration approach is used to obtain the exact solutions, since the problem size is small.  Table \ref{twoVSmul} summarizes the computational results of different instances that can be solved within one day. From Table \ref{twoVSmul}, we can see that cost percentage errors are below 20\% for all cases considered, which shows that the two-stage rolling horizon approach provides an acceptable approximation to the multi-stage problem.

\begin{table}[H]
	\centering
	\caption{Performance of two-stage rolling horizon\\
		in approximating multi-stage model }
	\label{twoVSmul}
	\begin{tabular}{@{}ccccc@{}}
		\toprule
		&            & multi-stage & \multicolumn{2}{c}{two-stage rolling horizon} \\ \midrule
		\textit{n}         & \textit{T} & cost        & avg. cost          & error \%         \\ \midrule
		\multirow{3}{*}{2} & 3          & 24.49       & 24.85              & 1.47\%           \\
		& 4          & 31.37       & 34.83              & 11.03\%          \\
		& 5          & 37.09       & 40.33              & 8.74\%          \\ \midrule
		\multirow{3}{*}{3} & 3          & 30.34       & 35.72              & 17.73\%          \\
		& 4          & 40.67       & 47.92              & 17.83\%          \\
		& 5          & 53.78       & 63.50              & 18.07\%          \\
	 \bottomrule
	\end{tabular}
\end{table}

\subsection{Case 1: degradation of wind turbine blades}
\label{sectionCase1}
Offshore wind farms are rapidly \cite{shafiee2015opportunistic} developing in recent years to provide the renewable energy for sustainable development. An offshore wind farm is usually built thousand meters away from the coastline and typically has hundreds of wind turbines. A wind turbine consists of multiple components, such as blade, main bearing, gearbox, and generator. If a maintenance team is sent to maintain a wind turbine, it is economically beneficial to jointly maintain other wind turbines \cite{tian2011condition}.

Due to the tensile mechanical loading and corrosive marine environment,  stress corrosion cracking (SCC) is one of the major contributors to blades' degradation. Shafiee et al. \cite{shafiee2015opportunistic}  model the monthly propagation of SCC as a stationary gamma process with an estimated shape parameter $\hat{\alpha}=0.542$ and rate parameter $\hat{\gamma}=1.147$. 

We consider a three-blade wind turbine system in this case study. Consider a planning horizon $T=10$ and an inspection interval of 12 months. Based on some pilot studies, we discretize the condition of a blade into 11 states, because this number of states provides us an acceptable decision accuracy while ensuring that the discretized states are robust to measurement errors. The PM cost is 200,000 Monetary Unit (MU). We consider two levels of CM costs: 600,000 MU and 1,000,000 MU.The setup cost is 130,000 MU and the failure threshold is $20$ cm \cite{shafiee2015opportunistic}.  

\begin{table}[h]  
	\centering  
	\caption{Results of multi-stage approximation (CM cost = 600,000)} 
	\label{TableCaseWind}
	\begin{tabular}{c|cc|cc|cc|c}
		\hline
		$t$\textbackslash{}$i$ & \multicolumn{2}{c|}{1} & \multicolumn{2}{c|}{2} & \multicolumn{2}{c|}{3} & \multirow{2}{*}{$\tilde{\xi}^\ast$} \\ \cline{2-7}
		& state & decision & state & decision & state & decision &  \\ \hline
		1 & 6 & no action & 5 & no action & \textbf{8} & \textbf{no action} & \multirow{5}{*}{8} \\
		2 & 10 & PM & 7 & no action & 11 & CM &  \\
		3 & 4 & no action & 10 & PM & 4 & no action &  \\
		4 & 5 & no action & 4 & no action & \textbf{8} & \textbf{no action} &  \\
		5 & 6 & no action & 9 & PM & 11 & CM &  \\ 
		6 & \textbf{8} & \textbf{no action} & 2 & no action & 2 & no action &\\
		7 & 10 & PM & 5 & no action & 5 & no action & \\
		8 & 6 & no action & \textbf{8} & \textbf{no action} & 10 & PM & \\
		9 & 8  & PM  & 10 & PM & 4 & no action & \\
		10 & 4 & no action & 4 & no action & 6 & no action & \\ \hline
		\multicolumn{8}{l}{\textit{Note: the decisions that are different from the decisions without }}\\
		\multicolumn{8}{l}{\textit{ economic dependence are shown in boldface}}\\	
	\end{tabular}
\end{table}

\begin{table}[h]  
	\centering  
	\caption{Results of multi-stage approximation (CM cost = 1,000,000)} 
	\label{TableCaseWind1}
	\begin{tabular}{c|cc|cc|cc|c}
		\hline
		$t$\textbackslash{}$i$ & \multicolumn{2}{c|}{1} & \multicolumn{2}{c|}{2} & \multicolumn{2}{c|}{3} & \multirow{2}{*}{$\tilde{\xi}^\ast$} \\ \cline{2-7}
		& state & decision & state & decision & state & decision &  \\ \hline
		1 & 6 & no action & 5 & no action & 8 & PM & \multirow{5}{*}{8} \\
		2 & 10 & PM & 7 & no action & 5 & no action &  \\
		3 & 4 & no action & 10 & PM & 8 & PM &  \\
		4 & 5 & no action & 4 & no action & 5 & no action &  \\
		5 & 6 & no action & 9 & PM & 8 & PM &  \\ 
		6 & 8 & PM & 2 & no action & 2 & no action &\\
		7 & 3 & no action & 5 & no action & 5 & no action & \\
		8 & 8 & PM & \textbf{7} & \textbf{PM} & 9 & PM & \\
		9 & 3  & no action  & 3 & no action & 4 & no action & \\
		10 & 6 & no action & 5 & no action & 6 & no action & \\ \hline
		\multicolumn{8}{l}{\textit{Note: the decisions that are different from the decisions without }}\\
		\multicolumn{8}{l}{\textit{ economic dependence are shown in boldface}}\\
	\end{tabular}
\end{table}

We use Algorithm 3 to solve this maintenance planning problem. We compare the decisions with and without considering economic dependence. Denote the PM threshold for each component in the two-stage model without considering economic dependence by $\tilde{\xi}^\ast$: If the component state is below $\tilde{\xi}^\ast$, no maintenance is performed, and if  the component is functioning and the state exceeds or equals to $\tilde{\xi}^\ast$, PM is performed.

 Tables \ref{TableCaseWind} and  \ref{TableCaseWind1} summarize the results when CM cost is 600,000 MU and 1,000,000 MU respectively. The threshold $\tilde{\xi}^\ast$ is 8 in both cases.  From Tables \ref{TableCaseWind} and  \ref{TableCaseWind1}, we can observe that maintenance decisions with and without consideration of economic dependence are different. For example, at the first decision stage ($t=1$) in Table \ref{TableCaseWind}, we can see that component 3 is not preventively maintained as it would be without considering economic dependence, so it can share the setup cost with component 1 at decision stage 2.  From Table \ref{TableCaseWind1}, we can see that component 2 is preventively maintained at decision period 8 when it is in state 7, which is below the optimal PM threshold when ignoring economic dependence.

\subsection{Case 2: degradation of crude-oil pipelines}\label{sectionCase2}

The reliability of crude-oil pipelines are critical to the safety of liquid energy supply in modern industries. Due to the corrosion, crack and mechanical damage, pipelines gradually deteriorate, which result in the decrease of pipeline wall thickness.

Based on the degradation data of six pipelines provided by a local chemical plant, we model the degradation process as a gamma process with random effects, where the shape parameter is $\alpha t$ and the rate parameter is $\gamma$. Random effects is used to capture the heterogeneities among all pipelines by assuming the rate parameter $\gamma$ follows a gamma distribution with shape parameter $\kappa$ and rate parameter $\lambda$. We regard the $\gamma$ as unknown for all pipelines, and use the expectation-maximization algorithm \cite{ye2014semiparametric} to estimate the parameters of $\alpha$, $\kappa$ and $\lambda$. Based on the data, we obtain the estimated parameters $\hat{\alpha}=1.0824$, $\hat{\kappa}=8.556$ and $\hat{\lambda}=7.654$.

\begin{table}[h]
		\centering  
	\caption{Components' states and maintenance decisions} 
	\label{TableCPipe1}
\begin{tabular}{c|cc|cc|cc|cc|cc}
	\hline
	$i$\textbackslash{}$t$ & \multicolumn{2}{c|}{1} & \multicolumn{2}{c|}{2} & \multicolumn{2}{c|}{3} & \multicolumn{2}{c|}{4} & \multicolumn{2}{c}{5} \\ \hline
	1 & 4 & no action & 6 & no action & \textbf{9} & \textbf{PM} & 5 & no action & 9 & no action \\
	2 & \textbf{8} & \textbf{PM} & 4 & no action & 5 & no action & \textbf{10} & \textbf{no action} & 11 & CM \\
	3 & 10 & PM & 4 & no action & \textbf{7} & \textbf{PM} & 5 & no action & 9 & no action \\
	4 & 5 & no action & \textbf{10} & \textbf{no action} & 11 & CM & 3 & no action & 5 & no action \\
	5 & 4 & no action & 5 & no action & \textbf{9} & \textbf{PM} & 5 & no action & 7 & no action \\
	6 & \textbf{9} & \textbf{PM} & 3 & no action & \textbf{6} & \textbf{PM} & 4 & no action & 6 & no action \\
	7 & 2 & no action & 3 & no action & 5 & no action & 7 & no action & 9 & no action \\
	8 & 10 & PM & 4 & no action & \textbf{7} & \textbf{PM} & 4 & no action & 6 & no action \\
	9 & \textbf{9} & \textbf{PM} & 3 & no action & \textbf{6} & \textbf{PM} & 3 & no action & 6 & no action \\
	10 & \textbf{9} & \textbf{PM} & 3 & no action & 4 & no action & 7 & no action & 8 & no action \\
	11 & 10 & PM & 3 & no action & 5 & no action & 6 & no action & 9 & no action \\
	12 & \textbf{7} & \textbf{PM} & 5 & no action & \textbf{8} & \textbf{PM} & 3 & no action & 4 & no action \\
	13 & \textbf{9} & \textbf{PM} & 2 & no action & 5 & no action & 7 & no action & 10 & no action \\
	14 & 1 & no action & 5 & no action & \textbf{9} & \textbf{PM} & 3 & no action & 4 & no action \\
	15 & 5 & no action & 6 & no action & \textbf{8} & \textbf{PM} & 4 & no action & 6 & no action \\
	16 & 10 & PM & 3 & no action & \textbf{8} & \textbf{PM} & 2 & no action & 5 & no action \\
	17 & 2 & no action & 3 & no action & 4 & no action & 7 & no action & 10 & no action \\ \hline
\end{tabular}
\end{table}

We consider 17 pipelines located in a small region and all pipes are shutdown when any pipe is maintained. For each pipeline, the wall thickness is $10mm$ when it is new, and the retirement thickness (failure threshold) $8mm$. The number of decision stages is 5. Suppose the costs of PM and CM are 5 and 20, and setup cost is 200. 

We solve this multi-stage pipeline maintenance problem by Algorithm 3.  We similarly compare the decisions with and without economic dependence. In this case, the optimal PM threshold without considering economic dependence is $\tilde{\xi}^\ast=10$. Table \ref{TableCPipe1} presents the state and maintenance action for each component $i$ at stage $t$. The decisions different from those without considering economic dependence are shown in boldface.  Because setup cost is much higher than the CM cost, from Table \ref{TableCPipe1}, we can see that there is a large number of different decisions, which shows the necessity of considering economic dependence when it exists.

We further investigate the impacts of parameter estimation uncertainty by considering estimated parameters as random variables. Let $\Theta=(\alpha,\kappa,\lambda)$  be the vector of parameters to be estimated. Given a parameter estimation $\hat{\Theta}$ the conditional cost is denoted by $C(\hat{\Theta})$ and the PDF value of $\hat{\Theta}$ is denoted by $f(\hat{\Theta})$. Based on the estimations, the unconditional cost is $C=\int C(\hat{\Theta})f(\hat{\Theta})d(\hat{\Theta})$. The form of PDF $f(\hat{\Theta})$ is typically complicated, and therefore it is difficult to derive the closed-form of $C$. We use the method in \cite{ye2012degradation} to approximate the unconditional cost. Specifically, we first use the bootstrap method to generate 500 samples of parameters. We then use the average conditional costs based on these samples to approximate the unconditional cost $C$. Our result shows that the mean and standard deviation of conditional costs are 770.57 and 260.5. Compared with the total maintenance cost 740 when parameter uncertainty is not considered, the impact of parameter estimation uncertainty is acceptable.

\section{Conclusion and future research}
\label{sectionConclusion}
In this paper, we study CBM optimization problem for multi-component systems over a finite planning horizon. We formulate the problem as a multi-stage stochastic integer program, providing analytical expressions for total cost and maintenance decisions. The proposed multi-stage stochastic maintenance optimization model has integer decision variables and non-linear transition probability due to the endogenous uncertainty, and is computationally intractable. We first investigate structural properties of the two-stage problem and design efficient algorithms to obtain high-quality solutions based on the structural properties. The multi-stage model is then approximated by the two-stage model using a rolling horizon approach. Computational studies show that Algorithm 1 can solve many cases to optimality quickly and Algorithm 2 can find high-quality solutions within a very short amount of time.

This work provides a new modeling approach in modeling  multi-component condition-based maintenance. Future research will consider other practical assumptions, such as the limit of maintenance budget, the requirement of system's reliability and availability, state-dependent PM cost, and state-dependent operational cost. In this paper, we mainly consider economic dependence, it is worth to further consider stochastic and structure dependences. It will also be interesting to address situations when we do not know the exact transition probabilities. A robust optimization approach may be applicable.

\newpage
\section*{Appendix}
\paragraph*{A.1. Proof of Proposition \ref{prop1}} 
\begin{proof}
	(1)	We first consider the case  where there is no failed component in the first stage.
	
	We need to compare the total costs among three cases for partition $(N_0,N_1)$: (a) $N_0 = \emptyset$, (b) $N_0 \neq \emptyset$ and $N_1 \neq \emptyset$ and (c) $N_0 = \mathcal{N}$. Denote $C_1$, $C_2$ and $C_3$ by the total costs for the three cases respectively, we show that $C_1$ is minimum.
	
	Denote the total cost for component $i \in \mathcal{N}$ \textit{without} considering economic dependence by
	\begin{displaymath}
	\left\{
	\begin{array}{lr}
	TC_{i}^{1}=c_{i,\text{pm}}+c_{\text{s}}+Q_{i}(1,m)(c_{i,\text{cm}}+c_{\text{s}}), & \tilde{x}_{i,1}=1,\\
	\specialrule{0em}{1ex}{1ex}
	TC_{i}^{0}=Q_{i}(g_{i,1},m)(c_{i,\text{cm}}+c_{\text{s}}), & \tilde{x}_{i,1}=0,
	\end{array}
	\right.
	\end{displaymath}
	 Because  $\tilde{x}_{i,1}^\ast=1$, we have $TC_i^1 < TC_i^0$, $\forall i \in \mathcal{N}$.
	
	Thus, we have
	\begin{displaymath}
	\begin{split}
	C_1 = &\sum_{i \in \mathcal{N}}TC_i^1-(n-1)c_\text{s}-c_\text{s}\sum_{i \in \mathcal{N}}Q_i(1,m)+c_\text{s}(1-\prod_{i \in \mathcal{N}}(1-Q_i(1,m)),\\
	C_2 = &\sum_{i \in N_0}TC_i^0+\sum_{i \in N_1}TC_i^1-(|N_1|-1)c_\text{s}-c_\text{s}(\sum_{i \in N_0}Q_i(g_{i,1},m)+\sum_{i \in N_1}Q_i(1,m))\\
	&+c_\text{s}(1-\prod_{i \in N_0}(1-Q_i(g_{i,1},m))\prod_{i \in N_1}(1-Q_i(1,m)))  \; \text{and}\\
	C_3 = &\sum_{i \in \mathcal{N}}TC_i^0-c_\text{s}\sum_{i \in \mathcal{N}}Q_i(g_{i,1},m)+c_\text{s}(1-\prod_{i \in \mathcal{N}}(1-Q_i(g_{i,1},m)).
	\end{split}
	\end{displaymath}
	
	(1a) Prove $C_1 < C_2$.
	
	Because
	\begin{displaymath}
	\left\{
	\begin{array}{lr}
	TC_i^0 > TC_i^1\\
	\specialrule{0em}{1ex}{1ex}
	\left(|N_1|-1\right)c_\text{s}+c_\text{s}\left(\sum_{i \in N_0}Q_i(g_{i,1},m)+\sum_{i \in N_1}Q_i(1,m)\right)<
	(n-1)c_\text{s}+c_\text{s}\sum_{i \in \mathcal{N}}Q_i(1,m)\\
	\specialrule{0em}{1ex}{1ex}
	c_\text{s}(1-\prod_{i \in N_0}(1-Q_i(g_{i,1},m))\prod_{i \in N_1}(1-Q_i(1,m)))>c_\text{s}(1-\prod_{i \in \mathcal{N}}(1-Q_i(1,m))
	\end{array}
	\right.
	\end{displaymath}
	we have  $C_1<C_2$.
	
	(1b) Prove $C_1 < C_3$
	
	It is easy to show that function $f(v_1,v_2,...,v_n)=\sum_{i \in \mathcal{N}}v_i+\prod_{i \in \mathcal{N}}(1-v_i)$ has $\frac{\partial{f}}{\partial{v_i}} \geq 0$ for all $0 \leq v_i \leq 1$, $i \in \mathcal{N}$. Therefore, we have $$\max(C_1)=C_1|_{Q_i(1,m)=0,\forall i \in \mathcal{N}}=\sum_{i \in \mathcal{N}}TC_i^1-(n-1)c_\text{s}$$ and $$\min(C_3)=C_3|_{Q_i(g_{i,1},m)=1,\forall i \in \mathcal{N}}=\sum_{i \in \mathcal{N}}TC_i^0-(n-1)c_\text{s}.$$ Because $TC_i^0 > TC_i^1$ for all $i \in \mathcal{N}$, we have $C_1 \leq \max(C_1) < \min(C_3) \leq C_3$. 
	
	Therefore, $C_1$ is minimum.

	(2) Consider the case where there exists at least one component failed at the first stage.
	
	Let set $N \subseteq \mathcal{N}$ collect all failed components and $N \neq \emptyset$. Following proof (1), we only need to compare case (a) and feasible case (b) because case (c) is not feasible. 
	
	The cost of case (a) and feasible case (b) are denoted by $C_1^\prime$ and $C_2^\prime$ respectively, where
	$$
	C_1^\prime = C_1 + \sum_{i \in N}(c_{i,\text{cm}}-c_{i,\text{pm}})
	$$
	and
	$$
	C_2^\prime = C_2 + \sum_{i \in N}(c_{i,\text{cm}}-c_{i,\text{pm}}).
	$$
	From $C_1<C_2$ in proof (1a), we have $C_1^\prime < C_2^\prime $.
	
\end{proof}

\paragraph*{A.2. Proof of Proposition \ref{prop2}}
\begin{proof}
	Denote the total cost for component $i \in \mathcal{N}$ \textit{without} considering economic dependence by
	\begin{displaymath}
	\left\{
	\begin{array}{lr}
	TC_{i}^{1}=c_{i,\text{pm}}+c_{\text{s}}+Q_{i}(1,m)(c_{i,\text{cm}}+c_{\text{s}}), & \tilde{x}_{i,1}=1,\\
	\specialrule{0em}{1ex}{1ex}
	TC_{i}^{0}=Q_{i}(g_{i,1},m)(c_{i,\text{cm}}+c_{\text{s}}). & \tilde{x}_{i,1}=0,
	\end{array}
	\right.
	\end{displaymath}
	and let $Q_i(1,m)=Q_i(1)$ and $Q_i(g_{i,1},m)=Q_i(g)$ $\forall i \in \mathcal{N}$, then we have
	\begin{displaymath}
	\begin{split}
	C=&\sum_{i \in N_0}TC_i^0 +\sum_{i \in N_1}TC_i^1-(\max(|N_1|-1,0))c_\text{s}-c_\text{s}(\sum_{i \in N_0}Q_i(g)  +\sum_{i \in N_1}Q_i(1))\\
	&+c_\text{s}(1-\prod_{i \in N_0}(1-Q_i(g))\prod_{i \in N_1}(1-Q_i(1)))
	\end{split}
	\end{displaymath}
	\begin{displaymath}
	\begin{split}
	C^\prime=&\sum_{i \in N_0^\prime}TC_i^0 +\sum_{i \in N_1^\prime}TC_i^1-(\max(|N_1^\prime|-1,0))c_\text{s}\\
	&-c_\text{s}(\sum_{i \in N_0^\prime}Q_i(g)  +\sum_{i \in N_1^\prime}Q_i(1))+c_\text{s}(1-\prod_{i \in N_0^\prime}(1-Q_i(g))\prod_{i \in N_1^\prime}(1-Q_i(1)))
	\end{split}
	\end{displaymath}
	If $N_1=\emptyset$, we have
	\begin{displaymath}
	\begin{split}
	C^\prime-C =& \sum_{k \in N}(TC_k^1-TC_k^0)+c_\text{s}\sum_{k \in N}(Q_k(g)-Q_k(1))\\
	&+c_\text{s}(\prod_{i \in N_0}(1-Q_i(g))\prod_{i \in N_1}(1-Q_i(1)) - \prod_{i \in N_0^\prime}(1-Q_i(g))\prod_{i \in N_1^\prime}(1-Q_i(1)))\\
	=&\sum_{k \in N}\underbrace{\left(c_{k,\text{pm}}-(Q_k(g)-Q_k(1)\right)c_{k,\text{cm}})}_{\rho_kc_\text{s}}+c_\text{s}\\&
	-c_\text{s}\underbrace{\prod_{i \in N_0}(1-Q_i(g))\prod_{i \in N_1}(1-Q_i(1))}_{p(N_0,N_1)}\underbrace{\left(\prod_{k \in N}\dfrac{1-Q_k(1)}{1-Q_k(g)}-1\right)}_{r_N}\\
	=&\sum_{k \in N}\rho_kc_\text{s}+c_\text{s}
	-c_\text{s}p(N_0,N_1)r_N
	\end{split}
	\end{displaymath}
	Therefore, from $C^\prime<C$, we have
	$$
	\dfrac{\sum_{k \in N}\rho_kc_\text{s}+c_\text{s}}{c_\text{s}r_Np(N_0,N_1)}=\dfrac{1+\sum_{k \in N}\rho_k}{r_Np(N_0,N_1)}=\Delta_{\text{r}}(N_0, N_1, N)<1.
	$$
	From $\Delta_{\text{r}}(N_0, N_1, N)<1$, we have $C^\prime<C$.

	Similarly, if $N_1 \neq \emptyset$,
	\begin{displaymath}
	\begin{split}
	C^\prime-C =& \sum_{k \in N}(TC_k^1-TC_k^0)-|N|c_\text{s}+c_\text{s}\sum_{k \in N}(Q_k(g)-Q_k(1))\\
	&+c_\text{s}(\prod_{i \in N_0}(1-Q_i(g))\prod_{i \in N_1}(1-Q_i(1)) - \prod_{i \in N_0^\prime}(1-Q_i(g))\prod_{i \in N_1^\prime}(1-Q_i(1)))\\
	=&\sum_{k \in N}\underbrace{\left(c_{k,\text{pm}}-(Q_k(g)-Q_k(1)\right)c_{k,\text{cm}})}_{\rho_kc_\text{s}}\\&
	-c_\text{s}\underbrace{\prod_{i \in N_0}(1-Q_i(g))\prod_{i \in N_1}(1-Q_i(1))}_{p(N_0,N_1)}\underbrace{\left(\prod_{k \in N}\dfrac{1-Q_k(1)}{1-Q_k(g)}-1\right)}_{r_N}\\
	=&\sum_{k \in N}\rho_kc_\text{s}
	-c_\text{s}p(N_0,N_1)r_N
	\end{split}
	\end{displaymath}
	Therefore, from $C^\prime<C$, we have
	$$
	\dfrac{\sum_{k \in N}\rho_kc_\text{s}}{c_\text{s}r_Np(N_0,N_1)}=\dfrac{\sum_{k \in N}\rho_k}{r_Np(N_0,N_1)}=\Delta_{\text{r}}(N_0, N_1, N)<1.
	$$		
	From $\Delta_{\text{r}}(N_0, N_1, N)<1$, we have $C^\prime<C$.
\end{proof}

\paragraph*{A.3. Proof of Proposition \ref{prop3}}
\begin{proof}
	Denote the total cost for component $i \in \mathcal{N}$ \textit{without} considering economic dependence by
	\begin{displaymath}
	\left\{
	\begin{array}{lr}
	TC_{i}^{1}=c_{i,\text{pm}}+c_{\text{s}}+Q_{i}(1,m)(c_{i,\text{cm}}+c_{\text{s}}), & \tilde{x}_{i,1}=1,\\
	\specialrule{0em}{1ex}{1ex}
	TC_{i}^{0}=Q_{i}(g_{i,1},m)(c_{i,\text{cm}}+c_{\text{s}}). & \tilde{x}_{i,1}=0,
	\end{array}
	\right.
	\end{displaymath}
	and let $Q_i(1,m)=Q_i(1)$ and $Q_i(g_{i,1},m)=Q_i(g)$ $\forall i \in \mathcal{N}$, then we have
	\begin{displaymath}
	\begin{split}
	C=&\sum_{i \in N_0}TC_i^0 +\sum_{i \in N_1}TC_i^1-(\max(|N_1|-1,0))c_\text{s}-c_\text{s}(\sum_{i \in N_0}Q_i(g)  +\sum_{i \in N_1}Q_i(1))\\
	&+c_\text{s}(1-\prod_{i \in N_0}(1-Q_i(g))\prod_{i \in N_1}(1-Q_i(1)))
	\end{split}
	\end{displaymath}
	\begin{displaymath}
	\begin{split}
	C^\prime=&\sum_{i \in N_0^\prime}TC_i^0 +\sum_{i \in N_1^\prime}TC_i^1-(\max(|N_1^\prime|-1,0))c_\text{s}-c_\text{s}(\sum_{i \in N_0^\prime}Q_i(g)  +\sum_{i \in N_1^\prime}Q_i(1))\\
	&+c_\text{s}(1-\prod_{i \in N_0^\prime}(1-Q_i(g))\prod_{i \in N_1^\prime}(1-Q_i(1)))
	\end{split}
	\end{displaymath}
	If $N_1^\prime = \emptyset$, we have
	\begin{displaymath}
	\begin{split}
	C-C^\prime =& \sum_{k \in N}(TC_k^1-TC_k^0)+c_\text{s}\sum_{k \in N}(Q_k(g)-Q_k(1))\\
	&+c_\text{s}(\prod_{i \in N_0^\prime}(1-Q_i(g))\prod_{i \in N_1^\prime}(1-Q_i(1)) - \prod_{i \in N_0}(1-Q_i(g))\prod_{i \in N_1}(1-Q_i(1)))\\
	=&\sum_{k \in N}\underbrace{\left(c_{k,\text{pm}}-(Q_k(g)-Q_k(1)\right)c_{k,\text{cm}})}_{\rho_kc_\text{s}}+c_\text{s}\\&
	-c_\text{s}\underbrace{\prod_{i \in N_0}(1-Q_i(g))\prod_{i \in N_1}(1-Q_i(1))}_{p(N_0,N_1)}\underbrace{\left(1-\prod_{k \in N}\dfrac{1-Q_k(g)}{1-Q_k(1)}\right)}_{s_N}\\
	=&\sum_{k \in N}\rho_kc_\text{s}+c_\text{s}
	-c_\text{s}p(N_0,N_1)s_N
	\end{split}
	\end{displaymath}
	From $C > C^\prime  $, we have $\sum_{k \in N}\rho_kc_\text{s}+c_\text{s}
	-c_\text{s}p(N_0,N_1)s_N > 0$. Therefore,
	$$
	\dfrac{\sum_{k \in N}\rho_kc_\text{s}+c_\text{s}}{c_\text{s}s_Np(N_0,N_1)}=\dfrac{1+\sum_{k \in N}\rho_k}{s_Np(N_0,N_1)}=\Delta_{\text{s}}(N_0,N_1,N)>1,
	$$
	From $\Delta_{\text{s}}(N_0,N_1,N)>1$, we have $C > C^\prime  $.
	
	Similarly, if $N_1^\prime \neq \emptyset$,
	\begin{displaymath}
	\begin{split}
	C-C^\prime =& \sum_{k \in N}(TC_k^1-TC_k^0)-|N|c_\text{s}+c_\text{s}\sum_{k \in N}(Q_k(g)-Q_k(1))\\
	&+c_\text{s}(\prod_{i \in N_0^\prime}(1-Q_i(g))\prod_{i \in N_1^\prime}(1-Q_i(1)) - \prod_{i \in N_0}(1-Q_i(g))\prod_{i \in N_1}(1-Q_i(1)))\\
	=&\sum_{k \in N}\underbrace{\left(c_{k,\text{pm}}-(Q_k(g)-Q_k(1)\right)c_{k,\text{cm}})}_{\rho_kc_\text{s}}\\
	&-c_\text{s}\underbrace{\prod_{i \in N_0}(1-Q_i(g))\prod_{i \in N_1}(1-Q_i(1))}_{p(N_0,N_1)}\underbrace{\left(1-\prod_{k \in N}\dfrac{1-Q_k(g)}{1-Q_k(1)}\right)}_{s_N}\\
	=&\sum_{k \in N}\rho_kc_\text{s}
	-c_\text{s}p(N_0,N_1)s_N 
	\end{split}
	\end{displaymath}
	From $C > C^\prime$, we have
	$$
	\dfrac{\sum_{k \in N}\rho_kc_\text{s}}{c_\text{s}s_Np(N_0,N_1)}=\dfrac{\sum_{k \in N}\rho_k}{s_Np(N_0,N_1)}=\Delta_{\text{s}}(N_0,N_1,N)>1,
	$$	
	From $\Delta_{\text{s}}(N_0,N_1,N)>1$, we have $C > C^\prime$
\end{proof}

\paragraph*{A.4. Proof of Proposition \ref{prop4}}
\begin{proof}

	We prove this proposition by showing that the cost of partition $(N_0^\ast,N_1^\ast)$ is no worse than that of any other feasible partitions.
	
	For any other feasible partition $(N_0^\prime,N_1^\prime)$ and the partition $(N_0^\ast,N_1^\ast)$ that is obtained by Algorithm 1, we always rewrite $(N_0^\prime,N_1^\prime)=(N_0 \cup N_b,N_1 \cup N_a)$ and  $(N_0^\ast,N_1^\ast)=(N_0 \cup N_a,N_1 \cup N_b)$ respectively, where set $N_0 = N_0^\prime \cap N_0^\ast$, $N_1 = N_1^\prime\cap N_1^\ast$, $N_b = N_0^\prime \backslash N_0=N_1^\ast \backslash N_1$ and $N_a = N_1^\prime\backslash N_1=N_0^\ast \backslash N_0$. We now show that the cost of partition $(N_0^\ast,N_1^\ast)$ is no worse than that of $(N_0^\prime,N_1^\prime)$ by the following three parts:  (1) When $N_b \neq \emptyset$, we have cost relationship  $(N_0^\ast,N_1^\ast) = (N_0 \cup N_a,N_1 \cup N_b) < (N_0 \cup N_a \cup N_b,N_1)$, (2) when $N_b \neq \emptyset$, we have cost relationship $(N_0 \cup N_a \cup N_b,N_1) < (N_0 \cup N_b,N_1 \cup N_a) =(N_0^\prime,N_1^\prime)$, and (3) we have cost $(N_0^\prime,N_1^\prime) = (N_0^\ast,N_1^\ast)$ if and only if $N_b = \emptyset$ and $\Delta_{\text{r}}(N_0 \cup N_a,N_1,N_a)=\Delta_{\text{s}}(N_0,N_1 \cup N_a,N_a)=1$.
	
	(1) When $N_b \neq \emptyset$, we have cost relationship  $(N_0^\ast,N_1^\ast) = (N_0 \cup N_a,N_1 \cup N_b) < (N_0 \cup N_a \cup N_b,N_1)$. 
	
	This is equivalent to show that given current partition $(N_0 \cup N_a \cup N_b,N_1)$, moving $N_b$ from the do-nothing set to the maintenance set can reduce cost. We next show that if we keep moving the component that arrives first in $N_b$ in Algorithm 1 to the maintenance set, the cost keeps reducing until $N_b = \emptyset$, which implies moving the whole set $N_b$ to the maintenance set reduces cost.
	
	Denote the costs of $(N_0 \cup N_a,N_1 \cup N_b)$ and $ (N_0 \cup N_a \cup N_b,N_1)$ by $C$ and $C_0$ respectively, and initialize $C^\prime = C_0$. We prove $C < C_0$  by the following steps:
	
	\textbf{Step 1}: If all components in $N_b$ are moved into $N_1^\ast$ after set $N_1$ does in Algorithm 1, then $C < C_0$ because the cost reduces if we repeat how Algorithm 1 moves $N_b$ to $N_1^\ast$.
	
	\textbf{Step 2}: In this step, there exists at least one component $i \in N_b$ that joins $N_1^\ast$ no later than some component in $N_1$. Suppose component $k \in N_b$ is the earliest one in $N_b$ that joins $N_1^\ast$ and suppose $k$ joins $N_1^\ast$ along with set $S^j$, i.e., $k \in S^j$, where $|S^j|=j$ and $S^j \subseteq N_1^\ast$. Therefore, when $S^j \subseteq N_1^\ast$ joins $N_1^\ast$, the current partition is $(N_0 \cup N_a \cup N_b \cup S, N_1 \backslash S)$, where set  $S^j\backslash \{k\} \subseteq S$, and hence from Proposition \ref{prop2}, we have 
	\begin{equation}
	\label{EquationA5_1}
	\Delta_{\text{r}}(N_0 \cup N_a \cup N_b \cup S, N_1\backslash S, S^j) < 1.
	\end{equation}
	
	\textbf{Step 3}: If $j=1$, then $S^j=\{k\}$. Denote the costs for partition $(N_0 \cup N_a \cup N_b\backslash \{k\}, N_1 \cup \{k\})$ by $C_1$. From Inequation (\ref{EquationA5_1}), we have $\Delta_{\text{r}}(N_0 \cup N_a \cup N_b, N_1, \{k\}) < \Delta_{\text{r}}(N_0 \cup N_a \cup N_b \cup S, N_1\backslash S, \{k\}) < 1$ and therefore $C_1 < C^\prime$. We then update $N_b=N_b\backslash \{k\}$ and $C^\prime = C_1$ and go to Step 1.
	
	\textbf{Step 4}: In this step, we have $j > 1$. From Algorithm 1, we know that any subset $N \subset S^j$ cannot join $N_1^\ast$ given current partition $(N_0 \cup N_a \cup N_b \cup S, N_1 \backslash S)$. From Proposition \ref{prop2}, we have 
	\begin{equation}
	\label{EquationA5_2}
	\Delta_{\text{r}}(N_0 \cup N_a \cup N_b \cup S, N_1 \backslash S, N)>1. 
	\end{equation}
	Let $N+\{k\}=S^j$. We have
	$$
	r_Np(N_0 \cup N_a \cup N_b \cup S, N_1 \backslash S)+\rho_k<\sum_{i \in N}\rho_i+\rho_k<r_{S^j}p(N_0 \cup N_a \cup N_b \cup S, N_1 \backslash S),
	$$
	where the first inequality is from Inequation (\ref{EquationA5_2}) and the second inequality is from Inequation (\ref{EquationA5_1}). Therefore, we have $\rho_k<(r_{S^j}-r_N)p(N_0 \cup N_a \cup N_b \cup S, N_1 \backslash S)$ and hence
	\begin{equation*}
	\begin{split}
	&\Delta_{\text{r}}(N_0 \cup N_a \cup N_b, N_1, \{k\})=\dfrac{\rho_k}{r_{\{k\}}p(N_0 \cup N_a \cup N_b, N_1)}\\
	&<\dfrac{(r_{S^j}-r_N)p(N_0 \cup N_a \cup N_b \cup S, N_1 \backslash S)}{r_{\{k\}}p(N_0 \cup N_a \cup N_b, N_1)}\\
	&=\dfrac{\prod_{i \in N}\frac{1-Q_i(1,m)}{1-Q_i(g_{i,1},m)}p(N_0 \cup N_a \cup N_b \cup S, N_1 \backslash S)}{p(N_0 \cup N_a \cup N_b, N_1)}\\
	&=\dfrac{p(N_0 \cup N_a \cup N_b \cup S\backslash N, (N_1\cup N) \backslash S)}{p(N_0 \cup N_a \cup N_b, N_1)}\leq 1,
	\end{split}
	\end{equation*}
	where the last inequality is from $N = S^j \backslash  \{k\} \subseteq S$. From Proposition \ref{prop2}, by denoting the cost of partition $(N_0 \cup N_a \cup N_b\backslash \{k\}, N_1 \cup \{k\})$ by $C_1$, we have $C_1<C^\prime$ since $\Delta_{\text{r}}(N_0 \cup N_a \cup N_b, N_1,\{k\})<1$. We then update $N_b=N_b\backslash \{k\}$ and $C^\prime = C_1$ and go to Step 1.
	
	Therefore,  we can always lower the cost $C^\prime$ by moving one component from $N_b$ to the maintenance set. When $N_b=\emptyset$, we have $C^\prime = C < C_0$.  
	
	(2) When $N_b \neq \emptyset$, we have cost relationship $(N_0 \cup N_a \cup N_b,N_1) < (N_0 \cup N_b,N_1 \cup N_a) =(N_0^\prime,N_1^\prime)$. 
	
	This is equivalent to show that moving set $N_a$ from the maintenance set to the do-nothing set can lower cost. From Proposition \ref{prop3}, we need to prove $\Delta_{\text{s}}(N_0 \cup N_b,N_1 \cup N_a,N_a) > 1$.
	
	By using the same method as proof (1), we can also prove the cost relationship $(N_0 \cup N_a,N_1 \cup N_b) < (N_0,N_1 \cup N_a \cup N_b)$ when $N_b \neq \emptyset$. From Proposition \ref{prop3}, we have $\Delta_{\text{s}}(N_0,N_1 \cup N_a \cup N_b,N_a)>1$. Therefore,
	\begin{equation*}
	\Delta_{\text{s}}(N_0 \cup N_b,N_1 \cup N_a,N_a)>\Delta_{\text{s}}(N_0,N_1 \cup N_a \cup N_b,N_a)>1
	\end{equation*}
	
	(3) We have cost $(N_0^\prime,N_1^\prime) = (N_0^\ast,N_1^\ast)$ if and only if $N_b = \emptyset$ and $\Delta_{\text{r}}(N_0 \cup N_a,N_1,N_a)=\Delta_{\text{s}}(N_0,N_1 \cup N_a,N_a)=1$.
	
	When $(N_0^\ast,N_1^\ast)  = (N_0^\prime,N_1^\prime) $, we have $(N_0 \cup N_a,N_1 \cup N_b) = (N_0 \cup N_a \cup N_b,N_1)=(N_0 \cup N_b,N_1 \cup N_a)$.
	
	The first equality $(N_0 \cup N_a,N_1 \cup N_b) = (N_0 \cup N_a \cup N_b,N_1)$ holds if and only if $N_b = \emptyset$. Otherwise, following the steps of proof (1), we can always have $(N_0 \cup N_a,N_1 \cup N_b) < (N_0 \cup N_a \cup N_b,N_1)$.
	
	Given $N_b = \emptyset$, the second equality is equivalent to $(N_0 \cup N_a,N_1)=(N_0,N_1 \cup N_a)$, which happens if and only if $\Delta_{\text{r}}(N_0 \cup N_a,N_1,N_a)=\Delta_{\text{s}}(N_0,N_1 \cup N_a,N_a)=1$ based on Corollary \ref{cor1}. 
	
\end{proof}

\paragraph*{B.1. Proof of Corollary 1:}
\begin{proof}
	We first show $p(N_0 \cup N_{\rm u}, N_1)r_{N} \leq p(N_0, N_1 \cup N_{\rm u})s_{N}$.
	\begin{displaymath}
	\begin{split}
	&p(N_0 \cup N_{\rm u}, N_1)r_{N} \\
	=&\prod_{i \in N_0 \cup  N_{\rm u}}(1-Q_i(g_{i,1},m))\prod_{i \in N_1}(1-Q_i(1,m))(\dfrac{\prod_{i \in N}(1-Q_i(1,m))-\prod_{i \in N}(1-Q_i(g_{i,1},m))}{\prod_{i \in N}(1-Q_i(g_{i,1},m))})\\
	=&\prod_{i \in N_0 \cup N_{\rm u}-N}(1-Q_i(g_{i,1},m))\prod_{i \in N_1 \cup  N}(1-Q_i(1,m))(\dfrac{\prod_{i \in N}(1-Q_i(1,m))-\prod_{i \in N}(1-Q_i(g_{i,1},m))}{\prod_{i \in N}(1-Q_i(1,m))})\\
	= &p(N_0 \cup N_{\rm u}\backslash N, N_1 \cup N)s_{N} \leq p(N_0, N_1 \cup N_{\rm u})s_{N},
	\end{split}
	\end{displaymath}
	where equality holds when $N=N_{\rm u}$.
	
	(1) When $N_1 \neq \emptyset$, we have 
	$$\Delta_{\text{r}}(N_0 \cup N_{\rm u}, N_1, N) =  \dfrac{\sum_{k \in N}\rho_k}{r_{N}p(N_0 \cup  N_{\rm u},N_1)} \geq \dfrac{\sum_{k \in N}\rho_k}{s_{N}p(N_0 ,N_1 \cup  N_{\rm u})} = \Delta_{\text{s}}(N_0, N_1 \cup N_{\rm u}, N),$$
	where equality holds when $N=N_{\rm u}$.
	
	(2) When $N_1 = \emptyset$ and $N \neq N_{\rm u}$, we have 
	$$\Delta_{\text{r}}(N_0 \cup N_{\rm u}, N_1, N) =  \dfrac{1+\sum_{k \in N}\rho_k}{r_{N}p(N_0 \cup  N_{\rm u},N_1)} > \dfrac{\sum_{k \in N}\rho_k}{s_{N}p(N_0 ,N_1 \cup  N_{\rm u})} = \Delta_{\text{s}}(N_0, N_1 \cup N_{\rm u}, N).$$
	
	(3)When $N_1 = \emptyset$ and $N = N_{\rm u}$, we have
	$$\Delta_{\text{r}}(N_0 \cup N_{\rm u}, N_1, N) =  \dfrac{1+\sum_{k \in N}\rho_k}{r_{N}p(N_0 \cup  N_{\rm u},N_1)} = \dfrac{1+\sum_{k \in N}\rho_k}{s_{N}p(N_0 ,N_1 \cup  N_{\rm u})} = \Delta_{\text{s}}(N_0, N_1 \cup N_{\rm u}, N).$$
	
	Therefore, $\Delta_{\text{r}}(N_0 \cup N_{\rm u}, N_1, N) > \Delta_{\text{s}}(N_0, N_1 \cup N_{\rm u}, N)$ when $N \subset N_{\rm u}$ and $\Delta_{\text{r}}(N_0 \cup N_{\rm u}, N_1, N) = \Delta_{\text{s}}(N_0, N_1 \cup N_{\rm u}, N)$ when $N=N_{\rm u}$.
\end{proof}

\clearpage
\bibliographystyle{ieeetr}
\bibliography{bibfile}
\end{document}